\documentclass[11pt]{article}
\usepackage[centertags]{amsmath}
  \usepackage{amsfonts}
\usepackage{amssymb}
\usepackage{makeidx}
\usepackage{newlfont}
\usepackage{amsthm} 
\usepackage{stmaryrd}
\usepackage[]{titletoc}  
\usepackage{mathrsfs}
\usepackage{stmaryrd}
 \usepackage{graphicx}  
 \usepackage{xypic}
\usepackage{hyperref}

\newlength{\defbaselineskip}
\newcommand{\setlinespacing}[1]%
           {\setlength{\baselineskip}{#1 \defbaselineskip}}

\theoremstyle{plain}
\newtheorem{thm}{Theorem}[section]
\newtheorem{cor}[thm]{Corollary}
\newtheorem{lem}[thm]{Lemma}
\newtheorem{prop}[thm]{Proposition}
\newtheorem{exam}[thm]{Example}
\newtheorem{rem}[thm]{Remark}

\makeatletter\@addtoreset{equation}{section} \makeatother
\begin{document}
\title{ Totally Abelian Toeplitz  operators and geometric invariants associated with their symbol curves}
\author{Hui Dan \quad Kunyu Guo \quad Hansong Huang }
\date{}
\maketitle \noindent\textbf{Abstract:} This paper mainly studies
totally Abelian operators
 in the context of analytic Toeplitz operators on both the Hardy and Bergman space.
  When the symbol is a meromorphic function on
 $\mathbb{C}$, we establish the connection between  totally Abelian property of these operators
  and  and geometric properties of their symbol curves.
   It is found that winding numbers and
 multiplicities of   self-intersection of symbol curves   play  an important role in this topic.
 Techniques of  group theory, complex analysis, geometry  and operator theory are intrinsic in this paper.
 As a byproduct,  under a mild condition we   provides an affirmative answer to a
question raised in \cite{BDU,T1}, and also construct some examples
to show that
  the answer is negative if the associated conditions are weakened.



\section{Introduction}
 ~~~~In this paper,   $\mathbb{D}$ denotes the unit disk in the complex plane $\mathbb{C}$,
  and $\mathbb{T}$ denotes the
 boundary of $\mathbb{D}$.
 Let $\mathfrak{M}(\overline{\mathbb{D}})$ consist of all meromorphic functions over $\mathbb{C}$ which
 have
  no pole  on the closed unit
  disk $\overline{\mathbb{D}}$, and
  $\mathfrak{R}(\overline{\mathbb{D}})$ denotes the set of all
  rational functions   without    pole on  $\overline{\mathbb{D}}$.
  It is clear that  $\mathfrak{M}(\overline{\mathbb{D}})\supseteq \mathfrak{R}(\overline{\mathbb{D}})  .$
 Let $H^\infty( \mathbb{D})$
 denote  the Banach algebra of all bounded holomorphic functions over $ \mathbb{D}$, and
 $H^\infty(\overline{ \mathbb{D}})$, as a subset of $H^\infty( \mathbb{D})$, consists of functions that are
 holomorphic on   $\overline{ \mathbb{D}}$.
  The Hardy space $H^{2}(\mathbb{D})$  consists of all holomorphic functions on $\mathbb{D}$
   whose Taylor coefficients at $0$
   are  square summable.   The   Bergman space $L_{a}^2 (\mathbb{D})$
    consists of
   all  holomorphic functions over $\mathbb{D}$ that are square integrable with respect to
  the normalized  area
   measure    over $\mathbb{D}.$
For each function $\phi $ in   $H^\infty(\mathbb{D})$,
  let $T_\phi$ denote  the Toeplitz operator  on the Hardy space $H^{2}(\mathbb{D})$
    or the   Bergman space $L_{a}^2(\mathbb{D})$  according to the context.

Let $H$  be a Hilbert subspace. For an operator $T$ in $ B(H)$,
$\{T\}'$ denotes the commutant of $ T $; that is,
$$\{T\}'=\{S\in B(H): ST=TS\},$$ which is a WOT-closed subalgebra of $B(H).$
The operator $T$ is called \emph{totally Abelian}  if $\{T\}'$ is
Abelian; equivalently, $\{T\}'$  is a maximal Abelian subalgebra of
$B(H)$ \cite{BR}. Berkson and Rubel \cite{BR} completely
characterized  totally Abelian operators in $B(H)$
 for $\dim H<\infty $: in this case  they proved that $T$ is totally Abelian if and only if $T$ has a cyclic  vector.
  In the case of  $\dim H=  \infty$ and  $  H$ being separable, they also
 characterized  when normal operators
  (including unitary operators) and non-unitary isometric operators are totally
  Abelian.
 Related work on analytic Toeplitz operators on $H^2(\mathbb{D})$
are also initiated by Berkson and Rubel \cite{BR}.
   It is shown that if $\phi   $  is an inner
function, then
   $T_\phi$ is totally Abelian on $H^2(\mathbb{D})$ if and only if
 there exist a unimodular constant $c$ and a point $\lambda\in
 \mathbb{D}$  such that
  $\phi(z)=c\frac{\lambda-z}{1-\overline{\lambda}z}$
     \cite[Theorem 2.1]{BR}.
 Recall that  $\{T_z\}'=\{T_h: h\in H^{\infty} (\mathbb{D}) \}$ is maximal Abelian. It follows that
 an analytic Toeplitz operator $T_\phi$ is totally Abelian if and only if
 $\{T_\phi\}'=\{T_z\}'$ (this statement also holds on many function spaces, such as
 weighted Bergman spaces). But in general, it is hard to judge when $\{T_\phi\}'=\{T_z\}'$  holds.
   Thus for a
 generic symbol $\phi\in H^\infty(\mathbb{D}),$ it is beyond touch to give a complete characterization for
totally Abelian property of $T_\phi.$
This leads us to consider the commutants for  analytic  Toeplitz
operators defined on the Hardy space  $H^2(\mathbb{D})$. In
\cite{DW}  Deddens and  Wong  raised several questions on this
topic.
 One of them asks   whether for each function $\phi\in  H^\infty(\mathbb{D})$,  there is
  an inner function $\psi$ such that $\{T_\phi\}'=\{T_\psi\}'$ and that
  $\phi=h\circ \psi$  for some $h\in  H^\infty(\mathbb{D})$.
     Baker, Deddens and  Ullman \cite{BDU} proved that for an entire function $\phi$,
 there is a positive integer $k$ such that $\{T_\phi\}'=\{T_{z^k}\}'$
 and $\varphi=h(z^k)$ for some entire function $h$.

For a function  $\phi$ in  $H^\infty(\mathbb{D})$, if there exists a
point $\lambda $ in $\mathbb{D}$ such that the inner part  of
$\phi-\phi(\lambda)$ is a finite Blaschke product, then $\phi$ is
called to be in Cowen-Thomson's class, denoted by $\phi\in
\mathcal{CT}(\mathbb{D})$. It is known that
$\mathcal{CT}(\mathbb{D})$ contains all nonconstant functions in
$H^\infty(\overline{\mathbb{D}})$. Below, $\mathcal{H}$ denotes the
Hardy space $H^2(\mathbb{D})$  or the  Bergman space $L_{a
}^2(\mathbb{D}) $.  As presented below is the remarkable theorem on
commutants for analytic Toeplitz operators, due to Thomson and Cowen
\cite{T1,T2,Cow1}; also see \cite[Chapter 3]{GH} for a detailed
discussion  and see \cite{DSZ,SZZ,GSZZ} for related work on this
line.
\begin{thm}\label{Tm1}[Cowen-Thomson]
  Suppose  $\phi\in   \mathcal{CT}(\mathbb{D})$. Then there exists
a finite Blaschke product $B$  and an $H^\infty$-function  $\psi$
such that $\phi =\psi(B)$ and $\{T_\phi\}'= \{T_B\}'$ holds on
$\mathcal{H} $.
\end{thm}
\noindent   The identity $\phi =\psi(B)$ in Theorem \ref{Tm1} is
called a \emph{Cowen-Thomson representation} of $\phi$.  Note that this $B$ is
of maximal order in the following sense:
 if there is  another finite Blaschke product $\widetilde{B}$ and
 a function $\widetilde{\psi}$ in  $H^\infty(\mathbb{D})$ satisfying
  $\phi =\widetilde{\psi}(\widetilde{B}),$ then
$$\mathrm{order} \, B \geq  \mathrm{order}\, \widetilde{B}.$$
  One defines a quantity  $b(\phi)$  to be
     the maximum of   orders of  $\widetilde{B}$, for which
there is a function $\widetilde{\psi}$ in  $H^\infty(\mathbb{D})$
such that
    $\phi =\widetilde{\psi}(\widetilde{B}) $, and  $b(\phi)$ is called \emph{ Cowen-Thomson order} of $\phi$.
Thus for the finite Blaschke product $B$ in Theorem \ref{Tm1} we
have order $B =b(\phi)$. Once $\phi$ is fixed,   it is not difficult
to show that $B$ is uniquely determined in the following sense.  If
there is another finite Blaschke product $B_0$ satisfying one of the
following:
\begin{itemize}
\item[(1)] order $B_0$=$b(\phi)$ and there is an $h\in H^\infty$ such that $\phi=h( B_0);$
\item[(2)]   $\{T_\phi\}'=\{T_{B_0}\}'$,
 \end{itemize}
 then there is a Moebius map $\eta$ such that $B_0=\eta( B)$.
 This means that \emph{Cowen-Thomson representation} of $\phi$ is unique in the sense of  modulo Moebius maps.

For convenience,  we now  omit the space $\mathcal{H}.$
  The following is an  immediate consequence of Theorem \ref{Tm1}, see \cite{BDU,T1}.
\begin{cor}\label{cor}
Let $\phi$ be a nonconstant  function in $H^\infty(\overline{\mathbb{D}})$. Then there exist  a finite
Blaschke product $B$ and a function $\psi$ in
$H^\infty(\overline{\mathbb{D}})$ such that $\phi =\psi(B)$ and
$\{T_\phi\}'= \{T_B\}'$ holds. If   $\phi$ is entire, then
$\psi$ is entire and $B(z)=z^n$ for some positive integer $n$.
\end{cor}

 Suppose $\phi $ belongs to Cowen-Thomson's class  $
\mathcal{CT}(\mathbb{D})$. Then $T_\phi$ is totally Abelian if and
only if $b(\phi)=1.$  When $\phi$ is an entire function, expanding
$\phi$'s Taylor series yields
 $$\phi(z)=\sum_{n=0}^\infty a_n z^n.$$
 Set $N=\gcd \{n:a_n\neq 0\}$, and then by Corollary \ref{cor} $\{T_\phi\}'=\{T_{z^N}\}'$. In this case,
$ b(\phi)= N. $
 Therefore, for a nonconstant entire function $\phi$, $T_\phi$  is totally Abelian if  and only if
 $$\gcd \,\{n:a_n\neq 0\}=1.$$

Therefore, for totally Abelian property of  analytic Toeplitz
operators $T_\phi$  it is important    to  determine Cowen-Thomson
order $b(\phi)$ of $\phi $,  and  it is of interest
  to determine  the exact form of the  Blaschke product   with order  $b(\phi)$.
  As we will see, there are several ways to study
 $b(\phi)$. The first attack is made by Baker, Deddens
and Ullman \cite{BDU} in the case of $\phi$ being an entire
function.
  In what follows, for
$c\not\in \phi(\mathbb{T})$, let wind $(\phi,c) $ denote    the winding
number of the curve $\phi(z)\,(z\in \mathbb{T})$ around the point
$c$.
    Write $n(\phi)$ for the number
$$ \min\, \{\mathrm{wind} (\phi, \phi(a)): a\in \mathbb{D}, \phi(a)\not\in \phi(\mathbb{T}) \}.$$
 For a function $\phi \in  H^\infty(\overline{\mathbb{D}})$, it is obvious that   $b(\phi)\leq n(\phi).$
If   $b(\phi)= n(\phi) $,  $\phi$ is
 called to \emph{satisfy Minimal Winding Number Property }(MWN Property).
It is shown  that  a nonconstant entire function $\phi$ enjoys MWN
Property \cite{BDU}. For functions in $  H^\infty(\overline{\mathbb{D}})$,  the problem  raised in \cite{BDU} and
\cite{T1} can be reformulated as: \vskip2mm \emph{if $\phi$ is a
nonconstant function in $ H^{\infty} (\overline{\mathbb{D}}),$ then
does $\phi$ have MWN Property; that is, $b(\phi)= n(\phi)$?} \vskip2mm

  If the answer is yes, for a large class of analytic Toeplitz operators  we can formulate their  totally Abelian
property   in terms of winding number.

 For those functions $\phi$ of MWN
Property in $ H^{\infty} (\overline{\mathbb{D}}) $ , we can  determine
  the exact form of $B $ appearing in Corollary \ref{cor}.
  To be precise, let $a$ be a point in $\mathbb{D}$ such that
  $\phi-\phi(a)$ does not vanish on $\mathbb{T}$ and
$$\mathrm{wind }\, (\phi, \phi(a))=n (\phi) = b(\phi).$$
Denote  the inner factor of $\phi-\phi(a)$ by $B_a$, and we  will
show that $B_a$ is the desired finite Blaschke product. For this,
let
$$\phi=\psi( B)$$ be the  Cowen-Thomson representation of $\phi$. By Corollary \ref{cor},
 $\psi$ is in $ H^{\infty} (\overline{\mathbb{D}}).$
 Let $$  \psi-\phi(a)=\eta F$$ be the   inner-outer factorization of $\psi-\phi(a)$, where $\eta$ is inner. We see that
  $$\phi-\phi(a)=(\psi-\phi(a))\circ B=\eta\circ B\,F\circ B.$$
 Therefore $B_a=c\, \eta\circ B$, where $c$ is a constant with $|c|=1.$  Since
 $$\mathrm{order}\, B_a=\mathrm{wind }\, (\phi, \phi(a))=b(\phi)=\mathrm{order}\, B,$$
 this forces that  $\eta$ to be a  Blaschke factor  of order $1.$  Therefore this  $B_a$ is
    the desired finite Blaschke products in Corollary \ref{cor}.
 In this way,     finding $B $  essentially
  reduces to finding one of these points $a$ (in general, such points $a$ consist of a nonempty open set).
 In some  cases of interest  this procedure is feasible (see Theorem
 \ref{ratgroup}).

 MWN Property is quite restricted.
 We will provide some examples of functions with good smoothness on the unit circle, and
  they are in Cowen-Thomson's class $\mathcal{CT}(\mathbb{D})$ but do not
 have MWN Property, see Examples \ref{36} and \ref{exam3}.
   It is known that entire functions have MWN Property \cite{BDU}.
 In Theorem \ref{ratgroup} we  extend this result to all nonconstant meromorphic functions in $\mathfrak{M}(\overline{\mathbb{D}})$.

Before continuing, we introduce the finite self-intersection
property (FSI property). To be precise, for a function  $\phi$ in
the disk algebra  $ A(\mathbb{D})$ and $\eta\in \mathbb{T}$, let
$N(\phi-\phi(\eta),\mathbb{T})$ denote  the  cardinality of the
set $$\{w\in \mathbb{T}:\,\phi(w)-\phi(\eta)=0\}.$$
   called \emph{the multiplicity of
 self-intersection of the curve} $\phi(z)\,(z\in \mathbb{T})$ at the point
 $\phi(\eta).$
Write
$$N(\phi)=\min\,\{N(\phi-\phi(\eta),\mathbb{T}):\eta\in \mathbb{T}\},$$
 called \emph{ the multiplicity of
self-intersection of the curve} $\phi(z)\,(z\in \mathbb{T})$. It is
not difficult to verify that $b(\phi)\leq N(\phi)$.  A function $\phi$ in $
A(\mathbb{D})$ is called to \emph{have FSI property } if  except for
a finite subset of $\mathbb{T}$ each
 point $\xi\in\mathbb{T}$ satisfies $N(\phi-\phi(\xi),\mathbb{T})=1$  \cite{Qu}.

For meromorphic functions in  $\mathfrak{M}(\overline{\mathbb{D}})$,
we have the following result.
\begin{thm} Suppose $\phi$ is a nonconstant function in $\mathfrak{M}(\overline{\mathbb{D}})$. The following are equivalent:
  \label{rational}
\begin{itemize}
\item[(1)] the Toeplitz operator $T_\phi$ is  totally Abelian;
\item[(2)]   $\phi$ has FSI property.
\item[(3)]   $N(\phi)=1.$
 \end{itemize}
\end{thm}

 For  a nonconstant function $\phi$ in $H^{\infty} (\overline{\mathbb{D}})$, let
 $\phi=\psi (B)$ be a Cowen-Thomson representation of $\phi$. If
      $\psi\in H^{\infty} (\overline{\mathbb{D}}) $ has FSI
property,  then $\phi$ is called to\emph{ have FSI-decomposable property}.
 Quine \cite{Qu} showed that
 each nonconstant polynomial  has FSI-decomposable property.
  In this paper  we prove that each nonconstant function in $\mathfrak{M}(\overline{\mathbb{D}})$ also
enjoys the same property (see Theorem \ref{FSI2}).

For the characterization of  geometric property of symbol curves, we
introduce the semigroup $G (\phi)$. Precisely, for each continuous
function $\phi$ on $  \mathbb{T} $  define $G (\phi)$ to be the set
of all continuous maps $\rho$ from $ \mathbb{T}$ to  $ \mathbb{T}$
satisfying $\phi(\rho)=\phi$.

For a finite Blaschke product $\phi$, $G(\phi) $  is a finite cyclic
group, and furthermore
 $\sharp (G(\phi))=$ order $\phi $ \cite{CC}.
For $\phi\in  H^{\infty} (\overline{\mathbb{D}})$, we have the following result.
 \begin{thm} Suppose $\phi$ is a nonconstant function in      \label{group}
$H^{\infty} (\overline{\mathbb{D}})$. Then  $G(\phi) $ is a finite
cyclic group.
\end{thm}
Let $o(\phi)$ denote the order of $G(\phi); $ that is,
$o(\phi)=\sharp G(\phi).$ We thus have four integer quantities for a
function $\phi$: $ o(\phi), $ $b(\phi),$ $n(\phi)$ and $N(\phi).$ We
will prove that if $\phi$ is in $ H^{\infty}
(\overline{\mathbb{D}})$,  then   $ b(\phi)\leq o(\phi)\leq n(\phi)$
and $ o(\phi) \leq N(\phi)$ (see Section 2).

  For a finite Blaschke product $B$,  $o(B)=$ order $B=n(B)=N(B)$.
 More generally, we will prove that each nonconstant meromorphic function  in $\mathfrak{M}(\overline{\mathbb{D}})$ enjoys
this property.
 \begin{thm} \label{ratgroup} Suppose $\phi$ is a nonconstant function in
$\mathfrak{M}(\overline{\mathbb{D}})$. Then $$
 n(\phi)=b(\phi)=o(\phi)=N(\phi).$$
\end{thm}
\noindent In particular, for $\phi\in
\mathfrak{M}(\overline{\mathbb{D}})$ the Toeplitz operator $T_\phi$
is totally Abelian if and only if $o(\phi)=1$; equivalently, the
identity map is  the only continuous map $\rho: \mathbb{T}\to
\mathbb{T}$ satisfying $\phi (\rho)=\phi.$

\vskip2mm

This paper is arranged as follows.
 Section 2 first
 provides some basic properties of the group $G(\phi)$ for $\phi$ in $ H^{\infty} (\overline{\mathbb{D}})$ and
 gives the proof of Theorem \ref{group}.  Section 3
  focuses on Toeplitz operators with meromorphic symbols, discusses the MWN property,
  $o(\phi)$, $N(\phi)$
  and Cowen-Thomson order $b(\phi)$ of $\phi$ for $\phi\in
  \mathfrak{M}(\overline{\mathbb{D}})$, and
   gives the proof of Theorem   \ref{ratgroup}.
 Section 4 first presents the proof of  Theorem
 \ref{rational}, and then give  further results on
   FSI and FSI-decomposable properties.
     Section 5 constructs some examples. On one hand, we give some  totally Abelian Toeplitz operators defined
  by symbols in $\mathfrak{M}(\overline{\mathbb{D}})  $. On the other hand,    some
examples   show that  conclusion of Theorem  \ref{rational} can fail even if the associated functions
    have good  smoothness on $\mathbb{T}$.

\section{The group $G(\phi)$ }
This section  provides some basic properties of $G(\phi)$.

For a function $\phi$   holomorphic on the closure of  a domain
$\Omega,$ $N(\phi,\Omega)$ or $N(\phi,\overline{\Omega})$ denotes
the number of  zeros of $\phi$ on $\Omega$
 or $\overline{\Omega}$ respectively, counting multiplicity.
The  winding number of $\phi$ is defined to be wind $ (\phi,0).$

The following shows that each member in  $G(\phi)$ has very strong
restriction.
\begin{lem} \label{GP} Suppose  $\phi$ is a nonconstant function
in $H^\infty(\overline{\mathbb{D}})$.
 Then every $\rho\in G(\phi)$ is an automorphism of $\mathbb{T}$ with winding number $1$.
\end{lem}
\begin{proof}
For each  $\rho\in G(\phi)$,  define
 $$\Lambda=\{t\in[0,2\pi):\phi'(e^{it})\phi'(\rho(e^{it}))=0\}.$$
 We will show that $\Lambda$ is a finite set.
 In fact, let $  \mathcal{Z}'$ denote the zero of $\phi'$ on $\overline{\mathbb{D}}$
  and put
  $$\mathcal{F}= \phi^{-1}(\phi(  \mathcal{Z}'))\cap \mathbb{T}.$$
Since $\phi\in H^\infty(\overline{\mathbb{D}})$,   $\mathcal{F}$ is
a finite set.
  If $\phi'(\rho(e^{it}))=0$, then $\rho(e^{it})\in \mathcal{F}.$
  Therefore, $$e^{it}\in \bigcup_{ \varsigma \in \mathcal{F}} \{z\in \mathbb{T}:
   \phi(\rho(z))-\phi(\varsigma)=0\}=\bigcup_{ \varsigma \in \mathcal{F}} \{z\in \mathbb{T}:
   \phi(z)-\phi(\varsigma)=0\}.$$
Since $\phi$ is  holomorphic on $\overline{\mathbb{D}}$ and
nonconstant,
 the right hand side is a union of
finitely many finite sets. Hence $\{t\in[0,2\pi):
\phi'(\rho(e^{it}))=0\}$ is a finite set, and so is $\Lambda$.

 Write $(0,2\pi) \setminus \Lambda=\bigcup_{k=0}^{n-1}(t_k,t_{k+1})$,
 where $0=t_0<t_1<\dots<t_n=2\pi$.
  Since $\rho$ is continuous, there exists a real continuous function $\theta$
  on   $[0,2\pi]$ such that $\rho(e^{it})=e^{i\theta(t)}$.
  Then $$\phi(e^{it})=\phi(e^{i\theta(t)}),$$
   the Inverse Function Theorem implies that $\theta$ is differentiable on $(0,2\pi) \setminus \Lambda$.
    Taking  derivatives of $t$  yields that
\begin{equation}\theta'(t)=\frac {e^{it}\phi'(e^{it})}{e^{i\theta(t)}\phi'(e^{i\theta(t)})}\ne0,\  t\in(0,2\pi)
\setminus\Lambda. \label{21}\end{equation} Hence for $0\le k\le
n-1$, $\theta$ is   strictly monotonic on each interval
 $(t_k,t_{k+1}) .$
 $ $
Since $\phi(\mathbb{T})$ is of zero area measure and
$\phi(\mathbb{D})$ is open, one can pick $\lambda\in\mathbb{D}$ such
that $\phi(\lambda)\notin\phi(\mathbb{T})$. By Argument Principle,
we have
\begin{align*}
N(\phi-\phi(\lambda),\mathbb{D}) &=\frac {1}{2\pi i}\int_{\mathbb{T}}\frac {\phi'(\xi)}{\phi(\xi)-\phi(\lambda)} \mathrm{d}\xi \\
&=\frac {1}{2\pi}\int_0^{2\pi}\frac {\phi'(e^{it})}{\phi(e^{it})-\phi(\lambda)} e^{it}\mathrm{d}t .
\end{align*}
\begin{align*}&=\frac {1}{2\pi}\sum_{k=0}^{n-1}\int_{t_k}^{t_{k+1}} \frac {\phi'(e^{it})}{\phi(\rho(e^{it}))-\phi(\lambda)} e^{it}\mathrm{d}t \\
&\!\!\!  \stackrel{\small{(\ref{21})}}{=}\frac {1}{2\pi}\sum_{k=0}^{n-1}\int_{t_k}^{t_{k+1}} \frac {\phi'(e^{i\theta(t)})}{\phi(e^{i\theta(t)})-\phi(\lambda)} e^{i\theta(t)}\theta'(t)\mathrm{d}t \\
&= \frac {1}{2\pi}\sum_{k=0}^{n-1}\int_{\theta(t_k)}^{\theta(t_{k+1})} \frac {\phi'(e^{i\theta})}{\phi(e^{i\theta})-\phi(\lambda)} e^{i\theta}\mathrm{d}\theta \\
&=\frac {1}{2\pi}\int_{\theta(0)}^{\theta(2\pi)} \frac {\phi'(e^{i\theta})}{\phi(e^{i\theta})-\phi(\lambda)} e^{i\theta}\mathrm{d}\theta \\
&=\frac {\theta(2\pi)-\theta(0)}{4\pi^2i}\int_{\mathbb{T}}\frac {\phi'(\xi)}{\phi(\xi)-\phi(\lambda)} \mathrm{d}\xi \\
&=\sharp \rho\cdot N(\phi -\phi(\lambda),\mathbb{D}),
\end{align*}
where $ \sharp \rho=\mathrm{ wind} \,(\rho,0) .$ Since
$N(\phi-\phi(\lambda),\mathbb{D})$ is a positive integer, we have
$\sharp \rho=1$. This implies that $\rho$ is surjective.

It remains to   show  that $\rho$ is injective.
 Otherwise, there exist two points $\xi_1$ and $\xi_2$ in $\mathbb{T}$
 such that $\rho(\xi_1)=\rho(\xi_2)= \eta$. Let $A$ be
  the set of zeros of $\phi-\phi(\eta)$ in $\mathbb{T}$,
  and $\rho|_A:A\rightarrow A$ is
  surjective as $\rho  $ is
  surjective.
   Since $A$ is a finite set, $\rho|_A$ is actually a bijection,
   which is a contradiction to $\rho(\xi_1)=\rho(\xi_2).$  The proof is complete.
\end{proof}

\begin{cor} Suppose that $\phi$ is  a nonconstant function
in $H^\infty(\overline{\mathbb{D}})$ and \label{uni} both  $\rho_1 $
and  $\rho_2$ belong to $G(\phi)$. If $\rho_1(\xi_0)=\rho_2(\xi_0)$
for some point $\xi_0\in\mathbb{T}$,
    then $\rho_1=\rho_2$.
\end{cor}
\begin{proof} Let $\lambda$ be
an arbitrary point   in $\mathbb{T} \setminus \{\xi_0\}$. Let
$A$ be the  zero set of $\phi-\phi(\lambda)$ in $\mathbb{T}$.
 Arrange the points of $\{\xi_0\}\cup A$
 in the anti-clockwise direction:
 $$\xi_0, \xi_1, \dots, \xi_{n_0}\, (n_0\ge1).$$
  It is clear that $\rho_1(\{\xi_0\}\cup A)=\rho_2(\{\xi_0\}\cup A)$.
  By Lemma \ref{GP}, both $\rho_1$ and $\rho_2$ are  automorphisms
  of $\mathbb{T}$ with winding number $\#\rho_1=\#\rho_2=1$. Thus,
   when $\xi$ moves   along $\mathbb{T}$ in the positive direction,
   the images   $\rho_1(\xi)$ and $\rho_2(\xi)$ run in the same direction.
  As $\xi$ goes from $\xi_0$ to $\xi_1$,   $\rho_1$ and $\rho_2$ must coincide at the point $\xi_1$.
   By induction, we have $\rho_1(\xi_k)=\rho_2(\xi_k)$,  $1\le k\leq n_0$.
   In particular, $\rho_1(\lambda)=\rho_2(\lambda)$. The proof is finished.
\end{proof}

 In the case of finite Blaschke products,  $G(\phi) $ is a finite
cyclic group \cite{CC}.  In what follows we will prove that this result  also is true for functions in $H^{\infty} (\overline{\mathbb{D}})$. Now we come to the proof of Theorem \ref{group} (=Theorem \ref{group2}), which is represented as below.
   \begin{thm} Suppose $\phi$ is a  nonconstant function in      \label{group2}
$H^{\infty} (\overline{\mathbb{D}})$. Then   $G(\phi) $ is a finite
cyclic group.
\end{thm}
\begin{proof}  By
Lemma 2.1,  $G(\phi)$ is   a group.
  We will show that
$G(\phi) $ is a  finite cyclic group. Note that for   a fixed  point
$\zeta\in\mathbb{T}$,
$\{z\in \mathbb{T}: \phi(z)=\phi(\zeta)\}$ is a finite set, and by
Corollary \ref{uni} $G(\phi) $ is a  finite group.

 Let $\xi_0$ be a point on $\mathbb{T}$, and let
$\xi_0,\cdots, \xi_{n_0-1}$ be all zeros of $\phi-\phi(\xi_0) $ on
$\mathbb{T}$ in the anti-clockwise direction. Then define
$\{\xi_j\}_{j=0}^\infty$ to be  the infinite sequence
$$\xi_0,\cdots, \xi_{n_0-1}; \xi_0,\cdots, \xi_{n_0-1};\cdots.$$
That is, for all  $j$
 $$\xi_j=\xi_{[j]}\, ,\, \mathrm{where} \quad j\equiv [j]\ \, (\!\!\!\!\!\mod
 n_0\,),$$where $0\leq [j]\leq n_0-1$
Let $d$ be the minimal positive integer $l$ for which there is a
member $\rho_0$ in $G(\phi)$ satisfying $\rho_0(\xi_0)=  \xi_l $.
By Lemma
\ref{GP} $\rho_0$ maps each circular arc $\widetilde{\xi_i\xi_{i+1}}  $ 
to  $\widetilde{\xi_{j}\xi_{j+1}}$ for some $j$, preserving the
orientation.  By continuity, if $\rho_0(\xi_0)= \xi_l $, then one
 has
\begin{equation}\rho_0(\xi_i)= \xi_{i+l} , \, 0\leq i\leq n_0-1 . \label{use}\end{equation}
By definition of $d$, there is a member $\tau$ in $G(\phi)$
satisfying $\tau(\xi_0)=  \xi_d$. To finish the proof of  Theorem
\ref{group2}, it suffices to show
 that for each member $\rho$  in $G(\phi)$, there is an integer $ m$ such that
  $\rho=\tau^m$(in the sense of composition).  Write $\rho(\xi_0)=  \xi_l $,
 and there are two integers $k\geq 0$ and $l_0 $   such that  $  0\leq l_0<d $
and
  $$l=kd+l_0.$$
 Letting $\sigma=\tau^{-k}\rho,$   we have
  $\sigma\in G(\phi) $, and by  (\ref{use}) $\sigma(\xi_0)=\xi_{l_0}$.
  By definition of $d$ we have $l_0=0$.
 By Corollary \ref{uni}  $\sigma=id,$ forcing  $\rho=\tau^k$ to complete the proof.  \end{proof} 
For two positive integers $m$ and $n$,  write $m|  n$ to denote that
$m$ divides $n$. By  Theorem \ref{group2}, if $\phi$ is a
nonconstant function in $H^{\infty} (\overline{\mathbb{D}})$,
$G(\phi) $ is a finite cyclic group.  If  $\phi=\psi( B)$ is the
Cowen-Thomson representation of $\phi$,
 then $G(B)$ is a subgroup of $G(\phi),$ and   hence $o(B)| o(\phi).$
Since $o(B)=b(\phi)$, we have $$b(\phi)\,|\, o(\phi).$$

The following  gives  some
 properties of
the order $o(\phi)$ of $ G(\phi)$.
 \begin{cor} Suppose $\phi$ is a nonconstant  function in
$H^{\infty} (\overline{\mathbb{D}})$. Then for each $\xi \in
\mathbb{T}$,  $ o(\phi) \, | N(\phi-\phi(\xi), \mathbb{T}) .$
\label{div} Besides, for each $a\in \mathbb{D}$, if $\phi(a)\not\in
\phi(\mathbb{T})$, then $$ o(\phi) \, | \mathrm{wind}\,
(\phi,\phi(a)).$$
\end{cor}
 \noindent  In particular,
 $$o(\phi)\,| N(\phi ) \quad \mathrm{and} \quad o(\phi)\,| n(\phi )  ,$$
 where $N(\phi)=\min \,\{ N(\phi-\phi(\xi), \mathbb{T}):\, \xi\in \mathbb{T} \} $ and
  $$n(\phi )= \min\, \{\mathrm{wind}\, (\phi, \phi(a)): a\in \mathbb{D}, \phi(a)\not\in \phi(\mathbb{T}) \}.$$
\begin{proof}
   Let us have a close look at the proof of Theorem \ref{group2}.  Fix $\xi_0\in \mathbb{T}$, and let
$\xi_0,\cdots, \xi_{n_0-1}$ be all zeros of $\phi-\phi(\xi_0) $ on
$\mathbb{T}$ in anti-clockwise direction. Let
  $d$ be the minimal positive integer $l$ so that there is a
member $\rho$ in $G(\phi)$ satisfying $\rho(\xi_0)=  \xi_l $, and
 $\tau$ denotes the generator in $G(\phi)$ satisfying \linebreak $\tau(\xi_0)=  \xi_d$.
 Then by Lemma  \ref{GP} we have  $$\tau(\xi_i)= \xi_{i+d} , \, 0\leq i\leq n_0-1 ,$$
 and   $d\,| n_0$.
 Write $$n_0=jd.$$ For $k>0,$  $\tau^k(\xi_0)= \xi_{kd} $.
 By Corollary \ref{uni} we have that $\tau^k $ is the identity map if and only if  $\xi_{kd} =\xi_0$.
Therefore $j$ is the minimal positive number $k$ such that $\tau^k $ is the identity map, and
then
  $$j=o(\phi).$$ Since $n_0=jd,$ $j\,| n_0 $; that is, $ o(\phi) \, | N(\phi-\phi(\xi_0), \mathbb{T}) .$
  The first statement is proved.

   For $  0\leq i\leq j-1  $ let $\gamma_i$ denote the positively oriented circular arc
  $ \widetilde{\xi_{id}\xi_{(i +1)d}},$  ($\xi_{jd}=\xi_0$).
  Then   $\tau^i(\gamma_0)=\gamma_i$,
  and $\phi (\tau^i)=\phi$.  Also noting that $\phi(\gamma_i)$ are closed curves,  we have
 $$\mathrm{wind}\, (\phi(\gamma_i ),\lambda )=   \mathrm{wind }
 \,(\phi( \gamma_0), \lambda), \lambda \in \mathbb{C}\backslash \phi(\mathbb{T}),0\leq i\leq j-1.$$
 Since  $$\mathbb{T}= \bigcup_{i=0}^{j-1} \gamma_i \quad (as \quad curves ),$$
  $$\mathrm{wind} \,(\phi(\mathbb{T}),\lambda)=j \cdot \,\mathrm{wind}\, (\phi(\gamma_0),\lambda),
  \lambda \in \mathbb{C}\backslash \phi(\mathbb{T}).$$
 Thus $ o(\phi) \, | \mathrm{wind}\, (\phi, \lambda).$ In particular, we have
 $$ o(\phi) \, | \mathrm{wind}\, (\phi,\phi(a))$$ for each $a\in \mathbb{D}$ such that
  $\phi(a)\not\in \phi(\mathbb{T})$.
The proof is complete.
 \end{proof}
 Recall that a Jordan curve in $\mathbb{C}$ is
the image of a continuous injective  map from the unit circle
$\mathbb{T}$ into $\mathbb{C}$. For $\phi\in
H^\infty(\overline{\mathbb{D}})$, in the case of $\phi(\mathbb{T})$
being a Jordan curve, we have the following.
\begin{prop}\label{25} Suppose $\phi\in H^\infty(\overline{\mathbb{D}})$ and its image on $\mathbb{T}$  is   a Jordan curve. Then there is a
univalent function $h$ on $\mathbb{D}$ and a finite Blaschke product $B$ satisfying $\phi=h(B).$
In this case, we have   $$n(\phi)=b(\phi)=o(\phi)=N(\phi)=
\mathrm{order}\,B.$$
\end{prop}
\begin{proof}Write $\Gamma=\phi(\mathbb{T})$, the image of $\mathbb{T}$ under $\phi$.
Then $\Gamma$ is a Jordan curve. We will prove that $\Gamma=
\partial \phi(\mathbb{D})$. For this, note that $\partial
\phi(\mathbb{D}) \subseteq \Gamma.$ Assume conversely that
 $\partial \phi(\mathbb{D}) \neq \Gamma.$ Since  $\Gamma$ is a Jordan curve,
  $\mathbb{C}\setminus \partial \phi(\mathbb{D})$ is connected. A  fact from topology states that
  a domain $\Omega$ in $\mathbb{C}$ is a component of $\mathbb{C}\setminus\partial \Omega.$
  Letting $\Omega=\phi( \mathbb{D}),$ we have $\Omega=\mathbb{C}\setminus \partial \phi(\mathbb{D})$.
  However, this can not happen since   $\mathbb{C}\setminus \partial \phi(\mathbb{D})$ is not bounded.
  Therefore, $\Gamma= \partial \phi(\mathbb{D})$.

The Jordan curve $\Gamma$ divides the complex plane $\mathbb{C}$ to
an interior region and an exterior region. By $\Gamma= \partial
\phi(\mathbb{D})$, we know that $\phi$ is a proper map and $
\phi(\mathbb{D})$ is the interior region of $\Gamma $, a simply
connected domain. Let $h$ be a conformal map from $\mathbb{D}$ onto
$ \phi(\mathbb{D})$, and write $\psi=h^{-1}(\phi).$ Then $\psi$
is a holomorphic proper map from $\mathbb{D}$ to $\mathbb{D}$. Hence
$\psi$ is a finite Blaschke product \cite[Theorem 7.3.3]{Ru1}. Also,
we have  $ \phi= h ( \psi) $   as desired.

By Caratheodory's Theorem, a conformal map from $\mathbb{D}$ onto a
Jordan domain $\Omega$ extends to a continuous bijection from
$\overline{\mathbb{D}}$ onto $\overline{\Omega}$. Thus
  $h$ is bijective on $\mathbb{T}$. Rewrite $\psi=B.$ By $ \phi=h(B)$,
 we get $o(\phi)=o(B)=\mathrm{order}\, B,$    $n(\phi)=n(B)$, and  $N(\phi)=N(B).$
 Since $B:\mathbb{T}\to \mathbb{T}$ is a covering map,
 $$o(B)=n(B)=  N(B) = \mathrm{order}\, B,$$
 forcing $o(\phi)=n(\phi) =N(\phi)= \mathrm{order}\, B.$
 Since  $  \mathrm{order}\,B \leq b(\phi)\leq o(\phi),$  we have
   $o(\phi)=n(\phi) =b(\phi)=N(\phi)= \mathrm{order}\,B.$ The proof is finished.
\end{proof}
\begin{rem} The last theorem in \cite{Cow2} says that  if $\phi: \mathbb{D}\to \phi(
\mathbb{D})$ is   an $n$-to-1 analytic map,
 then there is a finite Blaschke product $ B$ and
 a univalent function $h$ so that $\phi=h(B)$.
 Using this one can prove  the former part of Proposition \ref{25}.
   \end{rem}

\section{Toeplitz operators with   meromorphic  symbols}
~~~~This section focuses on the class $\mathfrak{M}(\overline{\mathbb{D}})$, consisting of
  all meromorphic  functions on $\mathbb{C}$ whose poles are outside   $\overline{\mathbb{D}}$. In this interesting case,
we give  the proof of Theorem    \ref{ratgroup}.
\subsection{Some preparations}
~~~~Before going on, let us introduce some notions and lemmas.

 We begin with the notion of analytic continuation
\cite[Chapter 16]{Ru2}. A function element is an ordered pair
$(f,D)$, where $D$ is an open disk and $f$ is a holomorphic function
on $D$. Two  function elements $(f_0,D_0)$  and $(f_1,D_1)$ are
called direct continuations if $D_0\cap D_1$ is not empty and $f_0
=f_1 $ holds on $D_0\cap D_1$. By a  curve,
 we mean a continuous map from $[0,1]$ into  $ \mathbb{C}$.
   Given a function element $(f_0,D_0)$ and  a curve    $\gamma$   with $\gamma(0)\in D_0$,
if there is  a partition of  $[0,1]$:
     $$0=s_0<s_1<\cdots <s_n=1$$
 and function elements $(f_j,D_j)(0\leq j \leq n)$ such that
   \begin{itemize}
  \item [1.]  $(f_j,D_j)$  and $(f_{j+1},D_{j+1})$
are direct continuations for all $j$ with\linebreak $ 0\leq j\leq
n-1 $;
  \item[2.] $\gamma [s_j,s_{j+1}]\subseteq D_j(0\leq j\leq n-1)$ and $\gamma(1)\in D_n$,
       \end{itemize}then
  $(f_n, D_n)$ is called \emph{an analytic continuation of $(f_0, D_0)$ along
     $\gamma$}.

\vskip2mm Suppose  $\Omega$ is a domain satisfying $D_0\cap
\Omega\neq \emptyset.$ A function element $(f_0, D_0)$ is called to
\emph{admit unrestricted continuation in}  $\Omega$ if for any curve
$\gamma$ in $\Omega$ such that
 $\gamma(0)\in D_0$,  $(f_0, D_0)$ admits an analytic
continuation along $\gamma$.
 Furthermore, analytic continuation along a curve is essentially unique; that is, if
 $(g ,U)$ is another  analytic continuation of $(f_0, D_0)$ along
     $\gamma$,  then on \linebreak
     $U\cap D_n$ we have  $f_n=g.$ We denote by $f_0(\gamma,s)$  the value of analytic continuation of $f_0$
      along $\gamma$
      at the endpoint  $\gamma(s)$ of
     $\gamma_s: t\mapsto \gamma(st), $   $0\leq t \leq 1.  $ In particular,
     $f_0(\gamma,1)=f_n(\gamma(1)).$

  \vskip1mm   For example, let $D=\{z\in \mathbb{C}: |z-1|<1\}$ and define
     $$f(z)=\ln z, z\in D$$
     with $\ln 1=0.$ Let $\gamma(t)=\exp (2 t\pi i).$ Then $(f,D) $ admits analytic continuation along $\gamma.$
     We have $f(\gamma,0)=f(1)=0,$ and in general
     $$ f(\gamma, t)=2 t\pi i, \, 0\leq t \leq 1.$$
Note that $f(\gamma, 1)=2\pi i\neq f(\gamma,0)$, but $\gamma(1)=
\gamma(0)$.
\vskip2mm

For a holomorphic function $f$ on a domain $V$, if there is a
subdomain $V$ of $U$ and a holomorphic function $\rho:V\to U$ such
that
$$f(z)=f(\rho(z)), z\in V,$$
then $\rho$ is called a local inverse of $f$ on $V$. The analytic
continuation of a local inverse of $f$ is also a local inverse.
\vskip2mm

Some notations will be frequently used. For each
$z\in\mathbb{C}\setminus \{0\}$, define
$$z^*= 1 / \overline{z}.$$
Let   $A$  be a subset of the complex plane, and define $A^* =
\{z^*:z\in A\setminus \{0\}\}.$ For a meromorphic function $f$
on domain $\Omega$,
 define $f^*$ by
 $$f^*(z)= (f(z^*))^*, \, z\in  \Omega^*.$$
 Note that $f^*$ is a meromorphic function if $f \not\equiv 0 $.

 In the sequel, we need the following lemma.
\begin{lem} Suppose that $f$ is a holomorphic function on a convex  domain $\Omega$ and
$x_0\in\Omega \cap \mathbb{R} $.            \label{real}
 If there exists a sequence $\{x_k\}$ in $\mathbb{R}\setminus \{x_0\}$  such
  that $x_k\to x_0(k\to\infty)$, and $f(x_k)\in\mathbb{R}$, then $f(\Omega\cap\mathbb{R})\subseteq\mathbb{R}$.
\end{lem}
\begin{proof}Write  $U=\Omega\cap\{\overline{z}:z\in\Omega\}$, which is itself a domain.
Then define $g(z)= f(z)-\overline{f( \overline{z} )}$
   on $U$.
 For each $k$, we have $$g(x_k)=f(x_k)-\overline{f(\overline{x_k})}
  =f(x_k)-\overline{(f(x_k))}=0.$$
 Since $x_0\in U$ and $x_0$ is the accumulation point of $\{x_k\}$,   $g\equiv0$.
In particular,  for each $x$ in $\Omega\cap\mathbb{R}$,
   $f(x)=\overline{f(\overline{x})}=\overline{f(x)}$. That is, $f(x)\in\mathbb{R}$ to finish the proof.
\end{proof}

Since there is a Moebius map mapping  the real line to the unit
circle,   we get a translation of the Lemma \ref{real}.
\begin{cor}
Assume $f$ is holomorphic in $O(\zeta_0,\delta)$ where
$\zeta_0\in\mathbb{T}$ and $\delta>0$. Suppose that there exists a
sequence $\{\zeta_k\}$ in $\mathbb{T}\setminus \{\zeta_0\}$ such
that $\zeta_k\to\zeta_0(k\to\infty)$  and $f(\zeta_k)\in\mathbb{T}$.
Then \label{arcarc}
$f(O(\zeta_0,\delta)\cap\mathbb{T})\subseteq\mathbb{T}$.
\end{cor}

  An observation is in order.  Let $f $ be a nonconstant
function
 holomorphic at $a$. By complex analysis, there is a neighborhood
 $W$ and a holomorphic function $\psi$ on $W$ such that
 $f (z)-f (a)= (z-a)^n \psi(z),\, z\in W$
 and $\psi(a)\neq 0.$ For enough small $W,$  $g(z)=(z-a)\sqrt[n]{\psi(z)}$ is univalent on  $W$,
 and we have
$$f (z)-f (a)=g(z)^n.$$ Furthermore, we can require
$W$ to be a Jordan domain such that $g(W)$ is a disk centered at
$0$. Therefore, we immediately get the following.
\begin{lem}\label{ext} Suppose  $f $ is a   nonconstant holomorphic function over a domain containing    both $a$
and $b$, and $f (a)= f (b).$  Then for each enough small  number
$\varepsilon>0$, there are two Jordan neighborhoods $W_1$ and $W_2$
of $a$ and $b$,  such that both
 $f |_{W_1}$ and $f |_{W_2}$ are proper maps onto $f (a)+\varepsilon \mathbb{D}.$

 Furthermore, in this case for each pair $(z,w)$  satisfying $f (z)= f (w)$, $z\in W_1\setminus\!\!\{a\}$, and
 $w\in W_2\setminus\!\! \{b\}$, there is a local inverse $\rho$ of $ f $ such that
 $\rho(z)=w$ and $\rho$ admits analytic continuation along any curve   in $W_1\setminus\!\! \{a\}$,
 with values in  $ W_2\setminus\!\! \{b\}$.
\end{lem}

\subsection{Proof of Theorem  \ref{ratgroup}}
 ~~~~ A problem raised in \cite{BDU} and \cite{T1} asks whether
    each nonconstant function in $ H^{\infty}
(\overline{\mathbb{D}}) $ has MWN Property. Under a mild condition
this is answered by    the following result (=Theorem
\ref{ratgroup}).
  \begin{thm} Suppose $\phi$ is a nonconstant function in      \label{ratgroup2}
$\mathfrak{M}(\overline{\mathbb{D}})$. Then $$ n(\phi)=b(\phi)=o(\phi)=N(\phi).$$ 
\end{thm}
   The proof of Theorem \ref{ratgroup2} is long and thus it
is divided into several parts. In what follows we will establish
some
   lemmas and corollaries and then prove Theorem \ref{ratgroup2}  at the end of this
   subsection.

\vskip2mm

 In this section, let $\phi$ be a nonconstant
function in $\mathfrak{M}(\overline{\mathbb{D}})$.
  We write $P$ for  the set
of poles of $\phi$ in $\mathbb{C}$, and $Z'$ for the set of zeros of
$\phi'$. Let $X=P\cup \phi^{-1}(\phi(0))\cup \phi^{-1}(\phi(Z'))$,
$Y=X\cup X^*$, and write
$$\widetilde{Y}=\phi^{-1}(\phi(Y)).$$  Note that  $\widetilde{Y}$ is  a countable set containing
$Y$ and $\phi^{-1}(\phi(\widetilde{Y}))=\widetilde{Y}$. Recall that
a planar domain  minus a countable  set is path-connected.
\begin{lem} Suppose that there is a point $\xi$ on  $\mathbb{T}\setminus \widetilde{Y}$ and a local inverse $\rho$
of $\phi$ at $\xi$ such that   $\rho=\rho^*$ on some neighborhood of
$\xi$.
  If $\gamma$ is a curve in $\mathbb{C}\setminus \widetilde{Y}$
such that $\gamma(0)=\xi$, then $\rho$ admits   analytic
continuation along $\gamma$. \label{ac1}
\end{lem}

\begin{proof}
To reach a contradiction, assume  that $\rho$ admits no analytic
continuation along $\gamma$.  Write $\gamma_s(t)=\gamma(st),\, t\in
[0,1]$ and put
$$s_0=\sup\,\{s\in[0,1]:\rho \ \mathrm{admits\ an\
analytic\  continuation\  along}\  \gamma_s\}.$$ Then it is clear
that $\rho$   admits no analytic continuation along $\gamma_{s_0}$;
otherwise there is some $s_1>s_0$ such that $\rho$   admits an
analytic continuation along $\gamma_{s_1}$ to derive a
contradiction. Recall that $\rho(\gamma, s)$ denotes the value of
the analytic continuation of $\rho$ at the endpoint $\gamma(s)$ of
  $\gamma_s$, where $0\leq s<s_0.$

  One will show that  $\{\rho(\gamma,s):s\in[0,s_0)\}$ is bounded.  Note that $ \rho(\gamma,s)$ is
  continuous in $s$. If  $\{\rho(\gamma,s):s\in[0,s_0)\}$ is not bounded, then these exists a sequence $\{s_n\}\subseteq[0,s_0)$ such that $\{s_n\}$
  tends to $s_0$, and \begin{equation}\lim_{n\rightarrow\infty}
  \rho (\gamma,s_n )=\infty.  \label{rat1}\end{equation}Since $\gamma \cap \widetilde{Y}=\emptyset, $  $\gamma$ has no intersection
  with $\phi^{-1}(\phi(0))$, and then the local inverse $\rho^*$ of $\phi$ admits an analytic continuation
  along $\gamma_s^*$, where
  $$\gamma_s^*(t)=(\gamma_s(t))^*,\, t\in[0,1]. $$
Then by (\ref{rat1})$$\lim_{n\rightarrow\infty} \rho^*
(\gamma^*,s_n)=0,$$ forcing
$$\lim_{n\rightarrow\infty}\phi
(\gamma^*(s_n))=\phi(0).$$ That is,
 $\phi(\gamma(s_0)^*)=\phi(0)$, and hence  $ \gamma(s_0)^* \in  \phi^{-1}(\phi(0))$.
      But   $\gamma$ has no intersection
  with the set $\phi^{-1}(\phi(0))^*$, which is a  contradiction. 
  Therefore  $\{\rho(\gamma,s):s\in[0,s_0)\}$ is bounded by a positive number $C$.
\vskip2mm
    Let $\{z_i\}_{i=1}^m$ be all the zeros of $\phi-\phi(\gamma(s_0))$ in $C\overline{\mathbb{D}}$.
    Since   $\gamma$ has no  intersection with $\phi^{-1}(\phi(Z'))$,
    we have that $\phi'(\gamma(s_0))\neq 0$ and $$\phi'(z_i)\neq 0, \,\, i=1,\ldots,m.$$   Then
one can find  a connected neighborhood $U$ of $\gamma(s_0)$  and
disjoint connected neighborhoods $U_i(i=1,\ldots,m)$ of  $z_i$  such
that $\phi|_U$ and $\phi|_{U_i}$ are  univalent. Since
$\phi(z_i)=\phi(\gamma(s_0))$ for $1\leq i\leq m,$ using Lemma
\ref{ext} and  contracting $U$ and $U_i$
 we have
that $$\phi(U_i)=\phi(U)=O(\phi(\gamma(s_0)),\varepsilon),\, 1\leq
i\leq m,$$for some $\varepsilon>0,$ and that
\begin{equation}\phi^{-1}\big(O(\phi(\gamma(s_0)),\varepsilon)\big)\cap C\overline{\mathbb{D}}
       \subseteq\bigsqcup_{i=1}^m U_i.  \label{rat2}\end{equation}
      By continuity of $\phi$, there exists   a positive number $\delta<s_0$
       such that $$\phi(\gamma[s_0-\delta,s_0])\subseteq
       O(\phi(\gamma(s_0)),\varepsilon).$$
      By (\ref{rat2})  $\{\rho(\gamma,s):s\in(s_0-\delta,s_0)\}$
       is a  connected set in $\bigsqcup_{i=1}^m U_i$, and thus  it is  contained in
       a single $U_j$ for some $1\leq j\leq m$. Letting   $$\tau= (\phi|_{U_j})^{-1}\circ (\phi|_{U}),$$
       we have that $\tau$ is a local
       inverse of $\phi$
       such that $\tau(\gamma(s_0))=z_j $ and $\tau(U )=U_j$.

       For each $s\in(s_0-\delta,s_0)$, let $\rho_s$ be the analytic
        continuation for $\rho$ along $\gamma_{s}$, and then $\rho_s $ is a direct continuation of
        $\tau$. Then by combining $\rho_s$ with $\tau$, we have that  $\rho$  admits analytic
        continuation along $\gamma_{s_0}$ to derive a   contradiction.
         The proof is   complete.
\end{proof}

\begin{lem}  \label{value}Suppose  $\phi\in \mathfrak{M}(\overline{\mathbb{D}})$  is not a rational function.
 Then for each positive number $C$, there exist two points $a$ and $a'$ in $\mathbb{C}$  such
 that
 $|\phi(a)|>C$, $|\phi(a')|<\frac 1C$ and $\min \{|a|,|a'|\} >C$.
\end{lem}
\begin{proof} Since $\phi\in \mathfrak{M}(\overline{\mathbb{D}})$, $\phi$ is a meromorphic function over $\mathbb{C}$,
 and then the
infinity  $\infty$ is either an isolated singularity  or the limit
of poles. If $\infty$ is
  an isolated singularity, $\infty$ is
a  removable singularity, a pole  or an essential singularity. If
 $\infty$ were either a  removable singularity or a pole, then
$\phi$ would have finitely many singularities (poles), and by
complex analysis  $\phi$ is a rational function. This is a
contradiction to our assumption. Therefore,  $\infty$ is an
essential singularity of $\phi$. By
 Weierstrass' theorem in complex analysis, for each point $w\in \mathbb{C}\cup \{ \infty\}$ there is a sequence $\{z_n\}$
  tending to $\infty$ such that $\{\phi(z_n)\}$ tends to $w.$ Hence the conclusion of Lemma
 \ref{value} follows.

 If  $\infty$  is the limit of
  poles of  $\phi$, then for a fixed number
$C>0,$ one can find a point $a$ satisfying $|a|>C$ and
$|\phi(a)|>C$. To complete the proof, we will show that there exists
a point $a'$ such that $|a'|>C$ and $|\phi(a')|<\frac{1}{C}$.
If this were not true, then we would have
$$\frac{1}{|\phi(z)|} \leq C, |z| >C,$$
where $\frac{1}{\phi(z)}$ equals zero if $z$ is one pole of $\phi$.
Since  $\frac{1}{\phi}$ is bounded at a neighborhood of $\infty,$
$\infty$ is a removable singularity of $\frac{1}{\phi}$. Then  $\frac{1}{\phi}$ has only finitely many poles in
$\mathbb{C}\cup \{\infty\}$. Then by complex analysis $\frac{1}{\phi}$
is a rational function, and so is $\phi$. This derives a
contradiction to finish the proof.
 \end{proof}

For    a nonconstant function $\phi$ in
$\mathfrak{M}(\overline{\mathbb{D}})$ and for a local inverse $\rho$
of $\phi$, let $\rho^-$ be the inverse of $\rho $. We will use
Lemmas \ref{ac1} and \ref{value}   to prove the following.
\begin{lem}\label{r15} Suppose   $ \phi \in\mathfrak{M}(\overline{\mathbb{D}})$ is not a rational function. Then
there exists a bounded domain $\Omega\supseteq\overline{\mathbb{D}}$
having the following property:  if  $\rho$ is a local inverse of
$\phi$ at a point $\xi \in \mathbb{T}\setminus \widetilde{Y}$ such
that $\rho=\rho^*$ on some neighborhood of $\xi$,
then for   each curve $\gamma$
in $\Omega\setminus  \widetilde{Y}$  with $\gamma(0)=\xi$, we have $ \rho (\gamma,1)\in\overline{\Omega}$, i.e. the
value of the analytic continuation $\widetilde{\rho}$ of $\rho$
along $\gamma$ at endpoint $\gamma(1)$ lies in $\overline{\Omega}$.
\end{lem}
\begin{proof}
Suppose   $ \phi \in\mathfrak{M}(\overline{\mathbb{D}})$ is not a
rational function. First we   give the construction of $\Omega$. By
comments above Lemma \ref{ext} there exists  a small neighborhood
$V$ of $0$ biholomorphic function $g:V\rightarrow r\mathbb{D}(r>0)$
such that $\phi(z)-\phi(0)=g(z)^k$ on $V$ for some positive integer
$k$. One can require that
 $V$ is a  Jordan domain and
$\partial V$ is contained in $\mathbb{D}$. Put
$$\Gamma=(\partial V)^*=\{\frac{1}{\overline{z}}:z\in \partial V \},$$which  is
a  closed Jordan curve outside  $\overline{\mathbb{D}}$. Let
$\Omega$  be the interior of $\Gamma$, and then
$$V^*=\mathbb{C}\setminus\overline{\Omega}.$$

Let $\xi\in\mathbb{T}\setminus  \widetilde{Y}$, and let $\gamma$ be
a curve  in $\Omega\setminus \widetilde{Y}$ with $\gamma(0)=\xi$.
Suppose that  $\rho$ is a local inverse of $\phi$ at   $\xi$
  such that $\rho=\rho^*$ on some neighborhood of $\xi$.   To reach a contradiction, we assume
  $\widetilde{\rho}(\gamma(1))=\rho (\gamma,1)\in \mathbb{C}\setminus\overline{\Omega}= V^*$.
  Let $$\gamma_s(t)=\gamma(st),\, t\in [0,1]$$ and by Lemma \ref{ac1}
   $\rho$ admits an  analytic continuation  $\rho_s$
    along $\gamma_s$. Recall that  $\rho(\gamma,s)$ is
    the value of  $\rho_s$ at the endpoint $\gamma_s(1)=\gamma(s),$
    and let
   \begin{equation}\sigma(s)=\rho (\gamma,s ), \,s\in [0,1]. \label{cont1} \end{equation}
Then $\sigma$ is a curve in $\mathbb{C}\setminus \widetilde{Y}$.
Since
  $\gamma$ has no intersection with
$\phi^{-1}(\phi(Z'))$, $\rho^{- }$ admits an analytic continuation
$\widetilde{\rho}^{- }$ along $\sigma$, and by (\ref{cont1}) we have
$$\rho^-(\sigma,t) =\gamma(t), \,t\in [0,1].$$
In particular, we get  \begin{equation}\widetilde{\rho}^{-
}(\sigma(1) ) =\rho^-(\sigma,1)=\gamma(1)   . \label{cont2}
\end{equation}

    Let $\{p_i\}_{i=1}^m$ be  all the poles of $\phi$ on $\overline{\Omega}$.
One can construct disjoint connected neighborhoods $U_i$
$(i=1,\ldots,m)$ of $ p_i$  such that
\begin{itemize}
\item[(1)] $\phi$ has no zeros in $\overline{U_i}$ for $1\leq i\leq m$;
 \item[(2)] $\overline{U_i}\cap \phi^{-1}(f(0))^*\subseteq\{p_i\}$ for $1\leq i\leq
 m$;
 \item[(3)] For such $i$ that $\phi(p_i^*)=\phi(0)$,
  there exists an enough small connected neighborhood $V_i\subseteq V$ of $ 0$,
 such that  $\phi|_{U_i^*}$, $\phi|_{V_i}$ are proper
 maps satisfying $\phi(U_i^*)=\phi(V_i)$; for other $i$, let $V_i=V.$
\end{itemize}
   In fact, Condition (1) is easy to fulfill.  Since $\phi^{-1}(\phi(0)) $ is discrete and
  $\phi^{-1}(\phi(0))^*$ has at most one accumulation point $0$,
   Condition (2) is fulfilled if we let $U_i$ be enough small. By Lemma \ref{ext}
  we can choose $U_i$ and $V_i$ to satisfy (3) and be as small as possible thus to meet (1) and (2).
   Therefore, one has (1)-(3) as desired.

Let $$M=\max_{z\in \overline{\Omega}\setminus\bigcup_{i=1}^m U_i}
|\phi(z)|,$$ and define
$$\varepsilon_i=\mathrm{dist}(\phi(0),\phi(\overline{U_i^*})),\, i=1,\cdots,m.$$
If each $\varepsilon_i$ equals  zero,  set $\varepsilon=+\infty$;
otherwise, write \begin{equation}\varepsilon =\min
\,\{\varepsilon_i: \varepsilon_i>0,\, 1\leq i\leq m\}
.\label{rat34}\end{equation}
 Then there exists a number $\delta>0$ such that
\begin{equation}\phi(\delta\mathbb{D})\subseteq O(\phi(0),\varepsilon) \quad  \mathrm{and}
\quad \delta\mathbb{D}\subseteq \bigcap_{i=1}^m V_i.
\label{rat33}\end{equation} By Lemma \ref{value} we get
 a point $a\notin \widetilde{Y}$ satisfying $$|a|>\frac{1}{\delta} \quad \mathrm{and}  \quad |\phi(a)|>M.$$
Since $V^*\setminus \widetilde{Y}$ is path-connected, we can choose
a curve $\varsigma$ in $V^*\setminus \widetilde{Y}$ connecting
$\widetilde{\rho}(\gamma(1))=\sigma(1)$ with $a$. By Lemma
\ref{ac1}, $\rho^{- }$ admits an analytic continuation $\tau$ along
$\sigma \varsigma$,  where $\sigma\zeta$ is defined by
\begin{equation*}
\sigma \varsigma (t)=\left\{ \begin{aligned}
            \sigma(2t) \quad & 0\leq t\leq \frac{1}{2}, \\
                    \varsigma (2t-1) \quad & \frac{1}{2}<t\leq 1.
                          \end{aligned} \right.
                          \end{equation*}
Then
 $\rho^{- *}$ admits an analytic continuation $\tau^*$
along $\sigma^*  \varsigma^*$. Note that both $\tau$ and
$\tau^*$ are local inverses of $f$. Since
$$|\phi(\tau(a))|=|\phi(a)|>M,$$
  by the
definition of $M$ we get either  $\tau(a)\in V^*$ or $\tau(a)\in
U_i$ for some $i$. As follows, we will distinguish two cases to derive contradictions.

\vskip2mm
\noindent \textbf{Case I.}  $\tau(a)\in V^*$.   For
$w,z\in V$, let $\phi(w)=\phi(z)$. Recall that on $V$ we have
$$\phi(z)-\phi(0)=g(z)^k,$$ and then $g(w)^k=g(z)^k.$ Since
$g|_V$ is  biholomorphic, we get
$$w=g^{-1}\circ(\lambda g(z)), \, z\in V,$$
  where
$\lambda=\exp(\frac{2\pi ji}{k})$ for some integer $j$ in $\{1,\cdots,k\}$. Rewriting
$\rho_{j}$ for the map \linebreak  $g^{-1}\circ(\lambda g(z))$, we
have
 $\rho_{j}(V)=V$ and $\rho_{ j}(0)=0$.

Note that $$\tau^*(a^*) =(\tau(a))^*  \in (V^*)^*=V. $$ Since
$\phi(\tau^*(a^*))=\phi(a^*)$ and $a^*\in
 V$, there exists a  $j_0$ such
that
 $$\rho_{j_0}(a^*)=\tau^*(a^*).$$
Therefore $\tau=(\tau^*)^*$  extends analytically to
$$ \rho_{j_0}^*:V^*\rightarrow V^*.$$
Recall that  $\tau$ and  $\widetilde{\rho}^{-}$ are  analytic
continuations of $\rho^{- }$ along $\sigma \varsigma$ and
$\sigma$, respectively. Thus $\widetilde{\rho}^{-}$ also extends
analytically to $ \rho_{ j_0} ^*$. Then by (\ref{cont2})
$$\gamma(1)=  \widetilde{\rho}^{-}(\sigma(1)) =
 \rho_{j_0}^*(\sigma(1))\in V^*.$$  This
contradicts  with the fact that   $\gamma \subseteq\Omega$.
\vskip2mm
\noindent  \textbf{Case II.}  There is some $i$ such that
$\tau(a)\in U_i$.
 First we   show $\phi(p_i^*)=\phi(0)$.  In fact, since $a^*\in
\delta\mathbb{D}$, by (\ref{rat33}) we have
$$|\phi(0)-\phi(\tau^*(a^*))|=|\phi(0)-\phi(a^*)|<\varepsilon. $$Since $\tau^*(a^*)=(\tau(a))^*\in U_i^*$,
$$\varepsilon_i=\mathrm{dist}(\phi(0),\phi(\overline{U_i^*}))<\varepsilon,$$
which along with (\ref{rat34}) gives $\varepsilon_i=0$.  This shows
that $\overline{U_i^*}\cap \phi^{-1}(\phi(0)) $ is not empty.
  By condition (2)  we immediately get
  $\overline{U_i^*}\cap \phi^{-1}(\phi(0)) \subseteq\{p_i^*\}$, and thus
  $$ \phi(p_i^*)=\phi(0) . $$
Since Condition (1) shows that $$\min \{ |\phi(z)|: z\in
\overline{U_i}, 1\leq i \leq m\}>0,$$ by Lemma \ref{value} there is  a point $a'\notin \widetilde{Y}$ satisfying
\begin{equation}
  |\phi(a')|<\min \{ |\phi(z)|: z\in \overline{U_i}, 1\leq i \leq m\}  \label{rat4} \end{equation}
   and $|a'|>\frac{1}{\delta}$.
By (\ref{rat33}) $a'\in V_i^*.$ Let $\zeta$ be a curve  in
$V_i^*\setminus \widetilde{Y}$ joining $a$ with $a'$, and let
$\widetilde{\tau}$ be the analytic continuation of $\tau$ along
$\zeta$. Since  both $\tau$ and $\tau^*$ are local
inverses of $\phi$, so are  $\widetilde{\tau}$ and
$\widetilde{\tau}^*$, and $\widetilde{\tau}^*$ is the analytic
continuation of $\tau^*$ along $\zeta^*$.  By Condition (3) and
Lemma \ref{ext}, along any curve in $V_i\setminus\{0\}$,  $\tau^*$
admits  analytic continuation  with values in
$U_i^*\setminus\{p_i^*\}$. Thus we have
$$\widetilde{\tau}^*(a'^*)\in U_i^*.$$ Since
$(\widetilde{\tau}(a'))^*=\widetilde{\tau}^*(a'^*)$,
$\widetilde{\tau}(a')$ lies in $U_i,$ and hence
$\phi(a')=\phi(\widetilde{\tau}(a'))\in \phi(U_i)$. This is a
contradiction to (\ref{rat4}). In either case, we conclude a
contradiction thus to finish  the proof of   Lemma \ref{r15}.
\end{proof}
  Suppose $\phi$ is a function in
$H^\infty(\overline{\mathbb{D}})$.
  For
$\xi\in\mathbb{T}$, define
$$m(\xi)=\lim_{\delta\to0^+}\min_{\eta\in
O(\xi,\delta)\cap\mathbb{T}}N(\phi-\phi(\eta),\mathbb{T}).$$
Clearly, $m(\xi)\le N(\phi-\phi(\xi),\mathbb{T})$. Write
$$S= \{\xi\in\mathbb{T}:m(\xi)<N(\phi-\phi(\xi),\mathbb{T})\}.$$
We need the following  lemma.
\begin{lem}\label{count}Let $\phi $ be a nonconstant function in $H^\infty(\overline{\mathbb{D}})$. Then $S$ is countable.
\end{lem}
\begin{proof}
To reach a contradiction, assume   $S$ is uncountable.  Let $\mathcal{Z}'$ denote the zero set of $\phi ' $ on $\overline{\mathbb{D}}$ and
 $$F=\phi^{-1}(\phi(\mathcal{Z}')).$$ For each
positive integer $j$, put $$S_j=\{\xi\in
S:N(\phi-\phi(\xi),\mathbb{T})=j\}.$$ Then there exists at least a
positive integer $l$ such that $S_{l}$ is uncountable. Recall that
an uncountable  set in $\mathbb{C}$ has infinitely many accumulation
points.
   One can pick an  accumulation point $\xi_0$ of $S_{l}$  such that
$\xi_0\not\in F$.

For $r\in (0,1)$, let $A_r$ denote  the   annulus $$\{z\in \mathbb{
C}: r<|z|<\frac{1}{r}\}.$$  Since $\phi-\phi(\xi_0)$ has finitely many zeros on  $\mathbb{T}$
and $\phi-\phi(\xi_0)$ is holomorphic on $\mathbb{T}$,
  one can pick an  $r(0<r<1)$ close to $1$ such that  all   zeros of $\phi-\phi(\xi_0)$
in $\overline{A_r}$ lie on  $\mathbb{T}$.  By
 Rouch$\mathrm{\acute{e}}$'s
  Theorem, there exists a positive number
$\delta $ such that for each $z$ in $ O(\xi_0,\delta)$
$$N(\phi-\phi(z),A_r)=N(\phi-\phi(\xi_0),A_r)=N(\phi-\phi(\xi_0),\mathbb{T})=l.$$ On
the other hand, there is a sequence $\{\xi_k\}$ in $  S_{l}\cap
[O(\xi_0,\delta)\setminus \{\xi_0\} ]$, such that
$\xi_k\to\xi_0(k\to\infty)$. Thus,$$l=N(\phi-\phi(\xi_k),A_r)\ge
N(\phi-\phi(\xi_k),\mathbb{T})=l.$$ This means that each zero of
$\phi-\phi(\xi_k)$ in $A_r$ lies on $\mathbb{T}$. Since $\xi_0\notin
F$, there exist  $l$ local inverses $\rho_0, \dots, \rho_{l-1}$ of
$\phi$ defined on $O(\xi_0,\delta)$;
 that is,   $$\phi (\rho_i)=\phi, \, i=0,\dots,l-1 .$$
  Note that
$\rho_0(\xi_0), \dots, \rho_{l-1}(\xi_0)$ are exactly $l$
zeros of $\phi-\phi(\xi_0)$ on $\mathbb{T}$, and   for  all 
 $k$  we have $$\rho_i(\xi_k)\in\mathbb{T},\, i=0,\dots,l-1.$$ By
Corollary \ref{arcarc}, $\rho_i(O(\xi_0,\delta
)\cap\mathbb{T})\subseteq\mathbb{T},\,i=0,\dots,l-1$. We can require
$\delta$ to be enough small such that $\rho_i(O(\xi_0,\delta
)\cap\mathbb{T})$ are  pairwise disjoint. Hence for each $\xi\in
O(\xi_0,\delta )\cap\mathbb{T}$, $\phi-\phi(\xi)$ has   $l$ distinct
zeros on $\mathbb{T}$, $\rho_0(\xi), \dots, \rho_{l-1}(\xi)$.
Therefore   $$m(\xi_0)\ge l= N(\phi-\phi(\xi_0),\mathbb{T}),$$ which
derives a contradiction to $\xi_0\in S$,  finishing the proof.
\end{proof}

 By using Lemmas \ref{r15} and \ref{count}, one  can prove the
following.
\begin{prop}  \label{rat3}Suppose $\phi$ is a nonconstant function in
$\mathfrak{M}(\overline{\mathbb{D}})$. Then $$ o(\phi)=b(\phi) .$$ 
\end{prop}
\begin{proof}
 We will first construct some  local inverses of
$\phi$ that maps some arc of $\mathbb{T}$ into  $\mathbb{T}$. For
this,  by Lemma \ref{count}
 $\phi(S)$ is countable as well as
$S$,
 and then
$\partial \phi(\mathbb{D})\setminus \phi(S)$ is  uncountable. Then
there is a point
  $\xi_0$ in $\mathbb{T}\setminus (\widetilde{Y}\cup S)$ satisfying
$$ \phi(\xi_0)\in\partial \phi(\mathbb{D}).$$ Rewrite
$n_0=N(\phi-\phi(\xi_0),\mathbb{T})=m(\xi_0)$. Then one can find an
$r \in (0,1)$ close to $1$ such that
 $$n_0=N(\phi-\phi(\xi_0),\mathbb{T})= N(\phi-\phi(\xi_0),\overline{A_{r}}).$$ By application of Rouch$\mathrm{\acute{e}}$'s theorem,
 there exists a positive number   $\delta>0$ satisfying
$$N(\phi-\phi(\xi), \mathbb{T})\leq N(\phi-\phi(\xi),\overline{A_{r }})= N(\phi-\phi(\xi_0),\overline{A_{r}})=n_0,\,  \xi\in
O(\xi_0,\delta)\cap\mathbb{T}.$$
  By  definition of
$m(\xi_0)$, $N(\phi-\phi(\xi),\mathbb{T})\geq n_0$, forcing
\begin{equation} \label{num} N(\phi-\phi(\xi),\mathbb{T})=m(\xi_0)=n_0,\,  \xi\in
O(\xi_0,\delta)\cap\mathbb{T}.\end{equation}  As done in Lemma
\ref{count} one can   find $n_0$ holomorphic functions $\rho_0,
\dots, \rho_{n_0-1}$ on $O(\xi_0,\delta )$ ($\delta$ can be
decreased
 if necessary) such that
\begin{itemize}
\item[(1)] for $z\in O(\xi_0,\delta)$, $$N(\phi-\phi(z),A_r)=N(\phi-\phi(\xi_0),A_r)=N(\phi-\phi(\xi_0),\mathbb{T})=n_0;$$
\item[(2)] $\phi(\rho_i)=\phi,\, 0\leq i\leq n_0-1 $;
\item[(3)] $\rho_i(O(\xi_0,\delta))\subseteq A_r,\, 0\leq i \leq n_0-1 $.
\end{itemize}
In particular,   $\rho_0(\xi_0), \dots, \rho_{n_0-1}(\xi_0)$ are
exactly those
 $n_0$ zeros of $\phi-\phi(\xi_0)$ on $\mathbb{T}$.
Then by (1) $$n_0=N(\phi-\phi(\xi),A_r)\ge
N(\phi-\phi(\xi),\mathbb{T})=n_0, \, \xi\in
 O(\xi_0,\delta)\cap\mathbb{T},$$ forcing all zeros of $\phi-\phi(\xi)$ in
$A_r$ to fall onto $\mathbb{T}$.  Hence by Conditions (2) and (3) we
get
$$\rho_i(\xi)\in\mathbb{T},\, i=0,\dots,n_0-1 .$$
Hence  there exists a neighborhood of $\xi_0$ where  we have
$\rho_i=\rho_i^*$
 for \linebreak  $  i=0,\dots,n_0-1, $  as they are equal on some arc of $\mathbb{T}$.
\vskip2mm
By Lemma \ref{ac1} for each curve $\wp$ in $\mathbb{C} \setminus
\widetilde{Y}$ such that $\wp(0)=\xi_0,$
  each member in $\{\rho_i :i=0,\dots,n_0-1\}$ admits
analytic continuation along  $\wp$. We will see that the family
$\{\rho_i :i=0,\dots,n_0-1\}$ is closed under analytic continuation.
For this,
 assume that       $\gamma  $ is a loop  in
$\mathbb{C}   \setminus \widetilde{Y}$ with $\gamma (0)=\gamma
(1)=\xi_0$.
  Let $\widetilde{\rho_i}\,( 0\leq i \leq n_0-1)$
 be the analytic continuation of $\rho_i$ along $\gamma$.
Clearly, all these $\widetilde{\rho_i}$ are  local inverses of
$\phi$, i.e.
   $\phi (\widetilde{\rho_i})=\phi$.
  Since $\phi(\xi_0)\in\partial \phi(\mathbb{D})$, $$ \widetilde{\rho_i} (\xi_0)\not\in \mathbb{D}.$$
   Besides, we have $\rho_i=\rho_i^*$ on some neighborhood of $\xi_0$,
    and then $$\phi ( \rho_i)=\phi (\rho_i^*)=\phi.$$  Write $\gamma^*(t)=(\gamma(t))^*\,(t\in[0,1])$ and
    define   $\widetilde{\rho_i}^*$ along $\gamma^*.$
         Hence $$\phi(\widetilde{\rho_i}^*)=\phi, \, 0\leq i\leq n_0-1.$$
     By similar reasoning as above, $ \widetilde{\rho_i}^*(\xi_0)\notin\mathbb{D}$. Also noting
     $ \widetilde{\rho_i} (\xi_0)\notin\mathbb{D}$
     gives $\widetilde{\rho_i}(\xi_0)\in\mathbb{T}$.
    Then it follows that
    $\{\widetilde{\rho_i}(\xi_0) :i=0,\dots,n_0-1\}$ is a permutation of
    $\{ \rho_i  (\xi_0):i=0,\dots,n_0-1\}$.
If two local inverses are equal at one point $\xi_0\not\in
\phi^{-1}(\phi(Z'))$, by the Implicit Function Theorem they
  are equal on a neighborhood of this point. Thus we have
  $$\{\widetilde{\rho_i} :i=0,\dots,n_0-1\}=\{ \rho_i  :i=0,\dots,n_0-1\}.$$
 Give two curves $\gamma_1$ and $\gamma_2$ with $\gamma_1(0)= \gamma_2(0)=\xi_0$ and $\gamma_1(1)= \gamma_2(1)$,
 $\gamma_1 \gamma_2^{- } $ is a loop with endpoints $\xi_0$.
Therefore, we have  that analytic continuations of  the family
$\{\rho_i :i=0,\dots,n_0-1\}$
  along $\gamma_1$ are the same as those along  $\gamma_2$. Thus
analytic continuations of  the family  $\{\rho_i :i=0,\dots,n_0-1\}$
does not depend on the choice of the curve.
  Define
 \begin{equation}B(z)=\prod_{i=0}^{n_0-1} \widetilde{\rho_i}(z), z\in  \mathbb{C} \setminus  \widetilde{Y}, \label{def}
 \end{equation}
  where we use analytic continuations. In what  follows, we will show that $B$ extends analytically to a finite Blaschke product and
  there are two cases to distinguish: $$\phi\in \mathfrak{R}(\overline{\mathbb{D}}) \quad \mathrm{or} \quad
   \phi \in  \mathfrak{M}(\overline{\mathbb{D}}) \setminus \mathfrak{R}(\overline{\mathbb{D}}).$$
\vskip2mm
 \noindent \textbf{Case I.}   $\phi \in \mathfrak{R}(\overline{\mathbb{D}})$. Thus   $\phi$ is a rational function,
  and then
  $\widetilde{Y}$ is a finite set. Assume that the infinity $\infty$ is a pole of $\phi$, without loss of
generality. Otherwise, one can compose $\phi$ with
 some $\eta\in \mathrm{Aut}(\mathbb{D}) $ defined by
 $$\eta(z)= \frac{ \alpha-z}{1-\overline{\alpha} z},$$
mapping  $\infty$ to a pole $ 1/ \overline{\alpha} $ of $\phi$.
Replacing $\phi$ with $\phi(\eta)$ reduces to the desired case.
Since $\phi \in \mathfrak{R}(\overline{\mathbb{D}})$, there is a constant $C_1>1$ such that $\phi$ is holomorphic on some neighborhood of
$C_1\overline{\mathbb{D}}$. Let $$M=\max \{|\phi(z)|:|z|\leq C_1\}
<+\infty.$$ Since $\phi(\infty)=\infty$,  there exists a constant
$C_2>0$ satisfying
$$|\phi(z)|>M,\,  |z|>C_2.$$
For $z\in C_1\mathbb{D} \setminus  \widetilde{Y} ,$ we have
$$|\phi(\widetilde{\rho_i}(z) )|= |\phi(z)|\leq M, 0\leq i\leq
n_0-1,$$ and then each $\widetilde{\rho_i}(z) $ is bounded by $C_2$.
Hence $B$ is  an analytic function  bounded by $C_2^{n_0}$.
Therefore, $B$ extends analytically to $C_1\mathbb{D}$ since $
\widetilde{Y} $ is a finite set.  Besides,   all
$\widetilde{\rho_i}(z) $ are unimodular on the circular arc
$O(\xi_0,\delta )\cap \mathbb{T}$, and so is $B.$ By Corollary
\ref{arcarc} $B$ is unimodular on $\mathbb{T}$, and hence
 $B$ is  a finite Blaschke product \cite{Ru1}.
 \vskip2mm

\noindent \textbf{Case II.}  $\phi \in  \mathfrak{M}(\overline{\mathbb{D}}) \setminus
\mathfrak{R}(\overline{\mathbb{D}}),$ then $\phi$ is not a rational
function. By Lemma
 \ref{r15}
there exists a bounded domain $\Omega\supseteq\overline{\mathbb{D}}$
such that for $z\in \Omega\setminus \widetilde{Y}$, we have $$
\widetilde{\rho_i}(z)\in \overline{\Omega}, \, 0\leq i \leq n_0-1.$$
For these $\rho_i$, each  analytic continuation  along a curve in
$\Omega\setminus \widetilde{Y}$ is defined by a chain of disks, and
hence by (\ref{def})   $B$ extends naturally to an open set
$V$($V\subseteq \Omega$) containing $\Omega\setminus \widetilde{Y}$.
Since $\widetilde{Y}$ is countable, there is a relatively closed
countable set $Y_0$ such that
$$V=\Omega \setminus Y_0.$$
Since $\Omega\setminus\! \widetilde{Y}$ is dense in $\Omega\setminus\!
Y_0$, for $z\in \Omega\setminus
Y_0$  we have
 $$ \widetilde{\rho_i}(z)\in \overline{\Omega}, \, 0\leq i \leq n_0-1.$$
 Therefore, $B$ is a well-defined bounded analytic function on $\Omega\setminus Y_0.$
Since $Y_0$ is a countable relatively closed set in $\Omega,$  $Y_0$
is $H^\infty$-removable, and thus
 $B$ extends analytically on $\Omega $.
 In particular, $B$ is   analytic on a neighborhood of $\overline{\mathbb{D}}$.
Since each $\widetilde{\rho_i}(z) $ is  unimodular on the circular
arc $O(\xi_0,\delta )\cap \mathbb{T}$,   so is $B.$ By Corollary
\ref{arcarc} $B$ is unimodular on $\mathbb{T}$, forcing
 $B$  to be a finite Blaschke product.

 In both cases we have shown that  $B$ extends analytically to a finite Blaschke product.
  All local inverses of $B$
are  exactly    $\{\widetilde{\rho_i} :i=0,\dots,n_0-1\}$, and clearly, 
 order $B=n_0.$   By   Corollary \ref{uni}, each member $\rho$ in $G(\phi)$
 is uniquely determined by the value $\rho(\xi_0)$.  Thus
$$o(\phi)\leq N(\phi-\phi(\xi_0), \mathbb{T}).$$
Note that $\rho_0(\xi_0), \dots, \rho_{n_0-1}(\xi_0)$ are  all
   zeros of $\phi-\phi(\xi_0)$ on $\mathbb{T}$, and thus
\begin{equation}o(\phi) \leq n_0= \mathrm{order}\, B=o(B). \label{last} \end{equation}
On the other hand, $\{\widetilde{\rho_i} :i=0,\dots,n_0-1\}$ are
local inverses of $\phi$, and then $B(z)\mapsto \phi(z)$ is a
well-defined analytic function, denoted by $h.$
 Since $h$ is bounded on $C_1\mathbb{D}$ minus
a finite set where $C_1>1$,    $h$ extends to a function in
$H^\infty(\overline{\mathbb{D}}),$ and
$$\phi=h( B).$$ This gives
$G(\phi)\supseteq G(B)$. Noting (\ref{last}), we have
$o(\phi)=o(B)$. By Section 2,
   $$o (\phi)\geq  b( \phi )  \geq \mathrm{order}\, B
      =o(B)=o (\phi) .$$
      forcing $b(\phi)=o(\phi)=n_0 $.
\end{proof}

 \vskip2mm To establish Theorem   \ref{ratgroup2}, we also need the following.
\begin{cor} \label{minimal}For a nonconstant  function $\phi\in \mathfrak{M}(\overline{\mathbb{D}})$,
 except
for a countable set  each  point $\xi$ in $\mathbb{T}$
 satisfies $N(\phi-\phi(\xi),\mathbb{T})= b(\phi).$
Furthermore, $N(\phi)=b(\phi).$
\end{cor}
\begin{proof} By Proposition \ref{rat3}, we
write
$$n_0= b(\phi)=o(\phi).$$
By Corollary \ref{div}  we have
$$N(\phi-\phi(\xi),\mathbb{T})\geq n_0 ,\, \xi\in \mathbb{T}.$$
Write
$$A=\{ \xi\in\mathbb{T}: N(\phi-\phi(\xi),\mathbb{T})> n_0\},$$and it suffices to show that $A$ is countable.
Assume conversely that      $A$ is uncountable. Since $A$ contains
uncountable accumulation points in itself, one can pick an
accumulation  point $\eta_0$ in   $ \mathbb{T}\setminus\!
(\widetilde{Y}\cup S)$.
  Write  $$l =N(\phi-\phi(\eta_0),\mathbb{T})   >n_0.$$

 In the first paragraph of the proof of Proposition
 \ref{rat3},
 by replacing $\xi_0$  with $\eta_0$
 we get
   $l$ local inverses on some neighborhood of $\eta_0$,
    which maps an arc in $\mathbb{T}$   into $\mathbb{T}$.
              Let $\gamma$ be a curve  in $\mathbb{C}\setminus  (\widetilde{Y}\cup S)$  connecting  $\eta_0$ and $\xi_0$, and
     these $l$ local inverses admit analytic continuations along $\gamma$,
     denoted   by $\widetilde{\tau_0}, \dots,\widetilde{ \tau_{l-1}}$.
      Also $\widetilde{\tau_0}^*, \dots, \widetilde{\tau_{l-1}}^*$
      are exactly analytic continuations  along $\gamma^*$ of the local inverses
      $ \tau_0^*, \dots,  \tau_{l-1}^*$ at $\eta_0$.
       For   $0\leq i\leq l-1$, neither $\widetilde{\tau_i}(\xi_0)$ nor $\widetilde{\tau_i}^*(\xi_0)$
        belongs to $\mathbb{D}$ as $\phi(\xi_0) \in \partial \phi(\mathbb{D})$.
Therefore  $\{\widetilde{\tau_i}(\xi_0):0\leq i\leq l-1 \}$ are $l$
distinct zeros of $\phi-\phi(\xi_0)$ on $\mathbb{T}$. But by
  (\ref{num}),
$$N(\phi-\phi(\xi_0),\mathbb{T})=n_0< l,$$ which derives a
contradiction. Hence $A $ is countable, as desired.  \end{proof}

Now we proceed to present the proof of Theorem \ref{ratgroup2}.
\vskip2mm \noindent \textbf{Proof of Theorem \ref{ratgroup2}.}
Suppose that $\phi$ is a nonconstant meromorphic function in
$\mathbb{C}$ without pole on $\overline{\mathbb{D}}.$ By Proposition
\ref{rat3} and Corollary \ref{minimal}, we have
 $$
N(\phi)=b(\phi)=o(\phi) .$$ It remains to show that
$$n(\phi)=b(\phi).$$
Recall that $$n(\phi)=\min_{z\in\mathbb{D},\phi(z)\not\in
\phi(\mathbb{T})}\mathrm{wind}\,(\phi(\mathbb{T}),\phi(z)).$$ By
Corollary \ref{div} $o(\phi)\leq n(\phi).$ Since $b(\phi)|o(\phi)$,
  $b(\phi)\leq n(\phi)$.
  It remains
to prove that
 $$n(\phi)\leq   b(\phi).$$
Recall that $Z'$ is the zero of $\phi$, and let $F=\phi^{-1}(\phi(Z'\cap \overline{\mathbb{D}})).$  By
Corollary \ref{minimal} there is a  point $w_0\in
\mathbb{T}\setminus F$ such that
    $\phi(w_0)\in\partial \phi(\mathbb{D}) $ and
   $$N(\phi-\phi(w_0),\mathbb{T})=b(\phi)  .$$ Since the zeros of $\phi$ are isolated in $\mathbb{C}$,
    there
    exists a positive constant   $ t>1 $  satisfying
    $$N(\phi-\phi(w_0),t \mathbb{D})=N(\phi-\phi(w_0),\overline{\mathbb{D}}) = b(\phi).$$
     By Rouch$\mathrm{\acute{e}}$'s Theorem, there is a positive number $\delta$
      such that $$N(\phi-\phi(z),t\mathbb{D})=N(\phi-\phi(w_0), t\mathbb{D})=b(\phi),\,z\in O(w_0,\delta). $$
Let $z_0$ be a point   in $O(w_0,\delta)\cap\mathbb{D} $  such that
$\phi(z_0)\not\in \phi(\mathbb{T})$, and by
 Argument Principle we get
  $$\mathrm{wind}\,(\phi(\mathbb{T}),\phi(z_0))=N(\phi-\phi(z_0),\mathbb{D})\leq N(\phi-\phi(z_0), t\mathbb{D})= b(\phi) .$$
  Thus $n(\phi)\leq   b(\phi)$, forcing $n(\phi)=b(\phi).$
This finishes the proof of
 Theorem \ref{ratgroup2}. $\hfill \square$

\section{FSI and FSI-decomposable  properties}
~~~~In this section it is shown that each nonconstant function in
$\mathfrak{M}(\overline{\mathbb{D}})$
 has   FSI-decomposable property. Based on this,
  the proof of Theorem  \ref{rational} is furnished.

\subsection{Proof of Theorem \ref{rational}}
~~~~Recall that $\mathfrak{M}( \overline{\mathbb{D}} )$ denotes the
class of all  meromorphic functions in $\mathbb{C}$ without pole on
$\overline{\mathbb{D}}$. One  main aim of this section is to prove
the following.
\begin{thm} For a nonconstant function $\phi\in \mathfrak{M}( \overline{\mathbb{D}} ),$
suppose  $\phi=\psi( B)$ is the  Cowen-Thomson representation of $\phi$. \label{FSI2}
Then  $\psi$ has FSI
property.
\end{thm}

 Later, by using Theorem \ref{FSI2} one will get
Theorem \ref{rational}, restated as follows.
\begin{thm} Suppose $\phi$ is a nonconstant function in $\mathfrak{M}(\overline{\mathbb{D}})$.
 The following are equivalent:
  \label{rational2}
\begin{itemize}
\item[(1)]the Toeplitz operator $T_\phi$ is  totally Abelian;
\item[(2)]   $\phi$ has FSI property.
\item[(3)]   $N(\phi)=1.$
\end{itemize}
\end{thm}

Recall that a point $\lambda$ in $\mathbb{C}$ is called a point   of
self-intersection of the curve $\phi(z)(z\in \mathbb{T})$  \cite{Qu}
 if there exist two distinct points  $w_1$ and $w_2$ on $\mathbb{T}$ such that
 $$\phi(w_1)=\phi(w_2)=\lambda; $$
 equivalently, $N(\phi-\lambda,\mathbb{T})>1.$
To prove Theorem \ref{FSI2}, we need the following.
\begin{lem} Suppose  $\phi\in H^\infty(\overline{\mathbb{D}})$.  \label{ncount}
Then the cardinality of points of self-intersections of the curve
$\phi(z)(z\in \mathbb{T})$  is either finite or $\aleph,$ the
continuum.
\end{lem}
\begin{proof} Suppose  $\phi\in H^\infty(\overline{\mathbb{D}})$. Denote the
set of all  points of self-intersection of  $\phi(\mathbb{T})$ by
$A$. If $A$ is finite,   the proof is finished.

  Assume
$A$ is an infinite set.  Then    $\phi^{-1}(A)\cap \mathbb{T}$ must
have an accumulation point
  $\xi_0$ on  $\mathbb{T}$. By the definition of points of self-intersection,  there is  a sequence
$\{\xi_k\} $ in $\mathbb{T}\setminus \{\xi_0\}$ and a sequence
$\{\eta_k\} $ in $\mathbb{T}$  such that \linebreak $\xi_k\to\xi_0\,
(k\to\infty)$, and
$$\phi(\xi_k)=\phi(\eta_k),\, \xi_k\neq\eta_k, \forall k.$$  Without loss of generality, one
assumes that $\{\eta_k\}$ itself converges to a point $\eta_0$ on
$\mathbb{T}$. Thus we have
  $$\phi(\xi_0)=\phi(\eta_0)\equiv\lambda_0.$$
Note that $\xi_0$ may be equal to $\eta_0$.

 Since $\phi $ is not constant,
by Lemma \ref{ext}
   there are two simply-connected
neighborhoods $U$ of $\xi_0$, $V$ of $\eta_0$ and a positive number
$\varepsilon$ such that $$\phi(U)=\phi(V)=\varepsilon\mathbb{D}
+\lambda_0,$$ and $\phi|_U,\phi|_V$ are holomorphic proper maps,
whose multiplicities are equal to the multiplicities of  zero of
$\phi-\lambda_0$ at   $\xi_0$ and $\eta_0$ respectively.
  Write $$\widehat{U}=U\setminus \{\xi_0\} \mathrm{\quad} \mathrm{and}
\mathrm{\quad}  \widehat{V}=V\setminus \{\eta_0\}.$$
  Let $N$ be the
multiplicity of the zero of $\phi-\lambda_0$ at the point $\eta_0$.  By
Lemma \ref{ext}, for each $z\in\widehat{U}$ we have the following:
\begin{itemize}
\item[(1)] there exist exactly $N$ distinct local inverses of $\phi$ on a connected neighborhood
$U_z $ of $z$ with values in $\widehat{V}$ and
$U_z\subseteq\widehat{U}$ ;
\item[(2)] each local inverse in (1) admits  analytic continuation along any curve
in $\widehat{U}$  starting from the point $z$.
\end{itemize}
Note that analytic continuation of a local inverse of $\phi$ in (1)
is also a local inverse, with values in $\widehat{V}$.

The following discussions are based on  the upper half plane $\prod$
rather than on the unit disk, and this will be more convenient.
  Let $\varphi$  be  a   Moebius
transformation   mapping  $\mathbb{D}$ onto  $\prod$,    its pole
being distinct from $\xi_0$ and $\eta_0$. Rewrite
$$x_k=\varphi(\xi_k) \quad \mathrm{and} \quad y_k=\varphi(\eta_k),\, k \geq  0.$$
Here by no means we indicate  that $x_k$ and $y_k$  are the real or
imaginary part of some complex number. Let $\delta $  be a   positive number
such that  \linebreak$\overline{O(x_0,\delta)}\subseteq\varphi(U)$,
and we define  four simply connected domains: $$D_0=\{z\in
O(x_0,\delta):Re(z-x_0)>0\},\, D_1=\{z\in
O(x_0,\delta):Im(z-x_0)>0\};$$
$$D_2=\{z\in O(x_0,\delta):Re(z-x_0)<0\},\, D_3=\{z\in
O(x_0,\delta):Im(z-x_0)<0\}.$$ Note that $D_1,D_2$ and  $D_3$ can be
obtained by a rotation of $D_1$. Since $\varphi^{-1}(D_i)$ is simply
connected for $ i=0,1,2,3 $, by (1) and (2) we get
  $N$   local inverses of $\phi$ with values in
$\widehat{V}$; and by the Monodromy Theorem these local inverse are
all analytic on $\varphi^{-1}(D_i)$ for fixed  $i$.
 Let $$\widetilde{\phi}=\phi\circ\varphi^{-1},$$ and we
  obtain $N$  local inverses of $\widetilde{\phi}$, which are analytic on  each domain
$D_i  $ for  $i=0,1,2,3$.  With no loss of generality,  assume there
are infinitely many  points of $\{x_k\}$ lying  in $D_0$. Then there
exists at least one local inverse $\sigma_0$ of $\widetilde{\phi}$
defined on $D_0$ so that  $\sigma_0$ maps $x_k$ to $y_k$, for
infinitely many
$k$. Define $$D_{4j+i}=D_i, \,  0\leq i\leq 3,\, j\in \mathbb{Z}_+.$$ 
Take  analytic continuations $(\sigma_i,D_i)(i=1,\cdots,4N)$ of
$(\sigma_0,D_0)$ along the chain $\{D_0,D_1,\dots,D_{4 N}\}$. Note
that $$D_0=D_4=D_8=\cdots,$$ and there are only finitely many
distinct local inverses of $\widetilde{\phi}$ on $D_0.$
 There must be a minimal positive integer $n_0\le N$ satisfying
$\sigma_{4n_0}=\sigma_0$. As follows,  we will use   function
elements $(\sigma_i,D_i)(i=0,\dots,4n_0-1)$ to construct a
holomorphic function on a disk $D.$ Precisely, write $D=
O(0,\sqrt[n_0]{\delta})$, and for $z\in D\setminus \{0\}$ define
$$\omega(z)=\begin{cases}
\sigma_0(z^{n_0}+x_0), & 0\le \arg z<\frac{\pi}{2n_0};\\
\sigma_1(z^{n_0}+x_0), & \frac{\pi}{2n_0}\le \arg z<\frac{\pi}{n_0};\\
\dots\\
\sigma_{4n_0-1}(z^{n_0}+x_0), & \frac{\pi}{2n_0}(4n_0-1)\le \arg
z<2\pi.\end{cases}$$ Then $\omega$ is
well-defined and holomorphic in $D\setminus \{0\}$. Observe that as
$z$ tends to $x_0$ in $D_i(i=0,\dots,4n_0-1)$, each $\sigma_i(z)$
tends to $y_0$. Therefore   $\omega$ is  bounded  near $0$, and
hence $0$ is a removable singularity of $\omega$. By setting
$\omega(0)=y_0$ we get a holomorphic function $\omega$ on $D$.

\vskip1mm

Since $\omega|_{D\cap\mathbb{R}^+}(x)=\sigma_0(x^{n_0}+x_0)$, and
$\sigma_0(x_k)=y_k$ holds for infinitely many $k$, we have
$\omega(\sqrt[n_0]{x_k-x_0})=y_k\in\mathbb{R}$ as   $x_k>x_0$. By
Lemma \ref{real},
 $$\omega(D\cap\mathbb{R})\subseteq\mathbb{R},$$ forcing
$\sigma_0(D_0\cap\mathbb{R})\subseteq\mathbb{R}$. Letting
$$\gamma=\varphi^{-1}(D_0\cap\mathbb{R})\subseteq \mathbb{T} ,$$ and

$$\widetilde{\sigma}_0(w)=\varphi^{-1}\circ\sigma_0\circ\varphi(w),
w\in \varphi^{-1}(D_0)\subseteq U, $$ we have
$\widetilde{\sigma}_0(\gamma)\subseteq\mathbb{T}$. Clearly
$\sigma_0$ is not  the identity map, and neither is
$\widetilde{\sigma}_0$. Let $$W=\{z\in\varphi^{-1}(D_0)\cap
\mathbb{T}: \widetilde{\sigma}_0(z)=z\},$$  and $W$ is at most
countable. Since the cardinality of $\phi(\gamma)$ is $\aleph,$ so
is $\phi(\gamma  \setminus W)$, finishing the proof of Lemma
\ref{ncount}.
\end{proof}

Now we are ready to give  the proof of Theorem \ref{FSI2}. \vskip2mm
\noindent  \textbf{Proof of Theorem \ref{FSI2}.}
  Suppose  that $\phi$ is a  nonconstant
   function in $\mathfrak{M}(\overline{\mathbb{D}}),$ and $\phi=\psi(B)$ is the Cowen-Thomson
    representation.   Corollary \ref{cor} says that $\psi$ is in $H^\infty(\overline{\mathbb{D}}).$
    By comments below Theorem \ref{Tm1}, $B$ is of maximal order and thus
  $\psi$ can not be written as a function of a
 finite Blaschke product of order lager than $1$. Again by Corollary \ref{cor} we have   $b(\psi)=1$.
    Corollary \ref{minimal} implies that
  $$   \{w\in \mathbb{T}:N(\psi-\psi(w),\mathbb{T})>1\}  $$
  is countable, as well as $\{\psi(w)\in \mathbb{T}:N(\psi-\psi(w),\mathbb{T})>1\}.$ But by Lemma  \ref{ncount}
   the cardinality of self-intersections of $\psi(z)(z\in \mathbb{T})$  is a natural number.  Thus
  $\psi$ has FSI property  as desired. $\hfill \square$
\vskip2mm
 We are ready to give the proof of Theorem \ref{rational2} (=Theorem \ref{rational}).
\vskip2mm \noindent \textbf{Proof of Theorem \ref{rational2}.} Note
(2) $\Rightarrow$ (3) is trivial. To show
  (3) $\Rightarrow$ (1),  assume $N(\phi)=1$. By Corollary \ref{minimal} we have
  $b(\phi)=N(\phi)=1$. Then   Theorem \ref{Tm1} gives that $T_\phi$ is  totally Abelian.

\vskip1mm
For (1) $\Rightarrow$ (2), let $\phi=\psi ( B) $ be a Cowen-Thomson representation. Then  by  Theorem \ref{FSI2}
  $\psi$ has FSI
property. Since $T_\phi$ is
totally Abelian,  $$\{T_z\}'= \{T_\phi\}' \supseteq \{T_B\}'
\supseteq \{T_z\}'.$$ Then $\{T_B\}'=
 \{T_z\}'$, forcing order $B=1.$ Since $\phi=\psi ( B) $, $\phi$  has FSI property as desire.
 The proof of Theorem
  \ref{rational2} is complete.
  $\hfill \square$

\vskip2mm It is straightforward to get  equivalent formulations for
(1)-(3) in Theorem
  \ref{rational2}: (4)
there is a point $\xi\in\mathbb{T}$ satisfying
  $ N(\phi-\phi(\xi),\mathbb{T})=1$; and (5)
    except for a countable or finite set every point $\xi\in\mathbb{T}$ satisfies
    $N(\phi-\phi(\xi),\mathbb{T})=1$.
 \vskip2mm
   Theorem \ref{FSI2} shows that each function $\phi$ in $\mathfrak{M}(\overline{\mathbb{D}}) $ has
   FSI-decomposable property; that is,
   for the  Cowen-Thomson
    representation $\phi=\psi(B)$, $\psi$ has FSI-property. In fact, we will see  that $\psi$ has quite special form
     (see Lemma \ref{rab}  and Theorem \ref{mero}).

\vskip2mm

  Recall that $\mathfrak{R}(\overline{\mathbb{D}}) $ consists of all rational functions which have
   no pole  on $\overline{\mathbb{D}}$. If $P$ and $Q$ are two co-prime polynomials,
  order $ \frac{P}{Q}$ is defined to be $\max\,\{\deg P,\deg Q\}$.   The following  is of
  independent interest.
\begin{lem}\label{rab} If $f$ is   in  $\mathfrak{R}(\overline{\mathbb{D}})$ and there is a function
$h$ on $  \mathbb{D}  $ such that
$$f=h( B),$$ where $B$ is a finite Blaschke product, then
$h$ is    $\mathfrak{R}(\overline{\mathbb{D}})$. In this case, we
have $\mathrm{order}\,f=\mathrm{order}\,   h \, \times
\mathrm{order} \,B $.
\end{lem}
\begin{proof} Suppose $f$ is a  function in $\mathfrak{R}(\overline{\mathbb{D}}) $ and   $h$ is a function on $  \mathbb{D}  $ satisfying
$$f=h( B),$$ where $B$ is a finite Blaschke product.
Let $n$=order $B$, and denote    $n$ local inverses of $B$ by
$\rho_0,\cdots,\rho_{n-1} $. Let $Z'$ denote the zero set of $B'$ in
$\mathbb{C}$, and by Bochner's
Theorem \cite{Wa} $Z'$ is a finite subset of $\mathbb{D}$.
Write $$ \mathcal{E}=B^{-1}(B(Z')).$$ It is known
that all local inverses admit unrestricted continuation in
$\overline{\mathbb{D}} \setminus \mathcal{E} .$ For each $ j (0\leq
j\leq n-1)$, define $\rho_j^*(z)= (\rho_j(z^*))^*, $ which admits
unrestricted continuation on $\mathbb{C} \setminus
\overline{\mathbb{D}}$ minus a finite set. Recall that the
derivative $B'$ of $B$ does not vanish on $\mathbb{T} $, $\rho_j $
is analytic on $\mathbb{T}$ and $\rho_j^*=\rho_j$ on $\mathbb{T}$.
Thus each $\rho_j$ admits unrestricted continuation on $\mathbb{C}$
minus a finite set, say $F_1$.

For each $z\in  \mathbb{D} ,$  by  $h (B(z))=f(z) $ we get
\begin{equation}f(z)=f(\rho_j(z)),\,  0\leq j\leq n-1.\label{equal0}\end{equation}
By analytic continuation, the above also holds for all $z\in
\mathbb{C} \setminus\!\! F_1.$ Let $P$ denote the poles of $f$,   a
finite set in $\mathbb{C}$. Then by (\ref{equal0})
  $B(z)\mapsto f(z)$ defines a holomorphic function $h$ on
$\mathbb{C} \setminus\!\! (B(F_1)\cup P )$. Thus on the complex
plane minus discrete points, we have
\begin{equation}f=h ( B). \label{equal}\end{equation}
If $f$ is a rational function, then its only possible isolated
singularities (including $\infty$) are poles. By (\ref{equal}), $h$
has at most finitely many singularities including $ \infty $,
 which are either removable singularity or
poles. Hence $h$ is a rational function. Since $f$ is holomorphic on
$\overline{\mathbb{D}}$, by (\ref{equal}) $h$ is bounded on a
neighborhood of $\overline{\mathbb{D}}$ with finitely many
singularities possible. Thus $h$ extends analytically on
$\overline{\mathbb{D}}$, forcing $h\in
\mathfrak{R}(\overline{\mathbb{D}})$.

\vskip1mm
Suppose $f$ is a rational function. Noting that $f$ can be written as
the quotient of two co-prime polynomials, by computations we have that $f$ is a
covering map   on $\mathbb{C} \setminus ( f^{-1}(f(\infty))\cup
\mathcal{E}),$ and the multiplicity is exactly
  order $f$. Since both $h$ and $B$ are rational functions, they can be regarded as covering maps on $\mathbb{C}$ minus
 some finite set. This  leads to the conclusion that $$\mathrm{order}
  \,  f=\, \mathrm{order}  \,  h  \times \, \mathrm{order}  \, B,$$
  to  complete  the proof.
\end{proof}
By Theorem \ref{FSI2} and Lemma \ref{rab}  we get the following.
\begin{cor}\label{48} Suppose  $R$ is a rational function in $\mathfrak{R}(\overline{\mathbb{D}})$ with prime order.
 Then  either $R$  is  a composition of a Moebius
 transformation and a finite Blaschke product, or $R$  has   FSI property. In the  later case, ¡¡$T_R$  is totally Abelian.
\end{cor}

Below we come to  functions in $
\mathfrak{M}(\overline{\mathbb{D}})\setminus
\mathfrak{R}(\overline{\mathbb{D}})  $.    The main theorem in \cite{BDU} says that each entire function $\phi$ has the Cowen-Thomson
    representation $\phi(z)=\psi(z^n)$ for some entire function $\psi$ and some integer $n$.  The following theorem   generalizes  the theorem in
\cite{BDU} to functions in $
\mathfrak{M}(\overline{\mathbb{D}})\setminus
\mathfrak{R}(\overline{\mathbb{D}})  $,  and is of independent interest.
\begin{thm} \label{mero} Suppose that $f\in\mathfrak{M}(\overline{\mathbb{D}})$  is not a rational function. Then
there is a positive integer $n$ and a function  $h  $ in
$\mathfrak{M}(\overline{\mathbb{D}})$ such that
  $$f(z)=h(z^n)$$ and $\{T_f\}'=\{T_{z^n}\}'$. Furthermore,
 $ n= o(f)=b(f)=n(f)=N(f).$ \end{thm}
\begin{proof} To prove Theorem  \ref{mero}, we begin with an
observation from complex analysis. By Lemma \ref{ext} and comments
above it, for a  function $\varphi$ holomorphic on a neighborhood of
$\lambda$, let $k=\mathrm{order}\, (\varphi,\lambda),$
 the multiplicity of  the zero  of $\varphi-\varphi(\lambda)$ at $\lambda$. Then there is a Jordan neighborhood $W$
of $\lambda$ such that
 $\varphi |_{W }$ is a $k$-to-1 proper map onto  a neighborhood of $\varphi(\lambda)$.
 This is right even if $\lambda=\infty$ or $\lambda$ is a pole of
 $\varphi $ (for $\lambda=\infty$, $ \mathrm{order}\, (\varphi,\infty)$$=\mathrm{order}\, (\varphi(1/z),0)$).

Based on this, we will show that if $B_0$
 is a finite Blaschke product of order $k$, and   order
 $(B_0,\infty)$ equals $k$, then $B_0$ is a function of
 $z^k.$ In fact, either $B_0(\infty)=\infty $ or $|B_0(\infty)|>1. $
If $|B_0(\infty)|>1,$  by letting $$ \psi (z)=
\frac{1/\overline{B_0(\infty)}-z }{1- z/B_0(\infty) }$$ we have
$\psi\circ B_0(\infty)=\infty.$ Then we can assume
$B_0(\infty)=\infty . $ Note that \linebreak  order  $(1/
B_0(1/z),0)=$ order $(B_0,\infty)=k.$ Write
$$B_0(z)= c\prod_{j=1}^k
\frac{\alpha_j-z}{1-\overline{\alpha_j}z},\, $$ where $ |c|=1$ and
$\alpha_j\in \mathbb{D}$ for all $j, $ and then
$$1/ B_0(1/z) = \overline{c} \prod_{j=1}^k \frac{\overline{\alpha_j}-z}{1-\alpha_jz}. $$
Since order  $(1/ B_0(1/z),0)=k$, $\alpha_j=0$ for all $j$. Then
$B_0$ is a function of $z^k.$ This fact will be used later.

Suppose that $f\in\mathfrak{M}(\overline{\mathbb{D}})$ is not a
rational function. By Theorem \ref{Tm1} there is   a function $h \in
H^\infty( \mathbb{D})$
  and a finite Blaschke product $B$ such that
  $$f(z)=h( B (z)), z\in \mathbb{D},$$
  and $\{T_f\}'=\{T_{B}\}'$. Without loss of generality, assume order $B=n\geq 2.$
 In the proof of Lemma \ref{rab} we have shown that there is a finite set $F_1$ such that
  each local inverse  $\rho_j $ of $B$
 admits unrestricted continuation on $\mathbb{C}\setminus F_1.$
 Let $P$ denote the poles of $f$.
By $B( \rho_j)=B$ on $\mathbb{C}\setminus F_1 $ and $f(z)=h(
B (z))$, we can define a holomorphic function
$$h(z):B(z) \mapsto f(z)$$
on $\mathbb{C}\setminus\!\! (B(F_1)\cup P )$. So $h$ has only
isolated singularities. Letting
$$F_0=\overline{F_1\cup B^{-1}( P )},$$ we have
\begin{equation}f(z)=h(B(z)), z\in \mathbb{C}\setminus F_0 \label{equal2}\end{equation}
If   order $(B,\infty)=n,$ then by the second paragraph $B$ is a
function of $z^n$, and hence
 $$\{T_f\}'=\{T_{B}\}'=\{T_{z^n}\}'.$$ By similar reasoning as (\ref{equal2}), there is a function
$\widetilde{h}$ such that $f(z)=\widetilde{h}(z^n)$ holds on
$\mathbb{C}$ minus a discrete set. In this case, it is
straightforward to show $\widetilde{h}$ is in
$\mathfrak{M}(\overline{\mathbb{D}})$. By Theorem \ref{ratgroup2},
we have $ n= o(f)=b(f)=n(f)=N(f) $ to complete the proof.

As follows, we assume that   order $(B,\infty)<n $  to reach a
contradiction. Since order $(B,\infty)<n $ and $B $ is an $n$-to-$1$
map,  we have   a point $a\in \mathbb{C}$, two neighborhoods
$\mathcal{N}_1$ of $a$ and $\mathcal{N}_2$ of $\infty$ such that $B(a)=B(\infty)$,
$B|_{\mathcal{N}_1}$  and $B|_{\mathcal{N}_2}$ are proper maps,
 and their images are equal. There are two cases
to distinguish: either $P$ is a finite set or $P$ is an infinite
set.
\vskip2mm
\noindent \textbf{Case I.}
  $P$ is a finite set. Then $F_0$ is a finite set. By
(\ref{equal2}), $f$ has similar behaviors at $a$ and at $\infty.$
Since $f$ is a meromorphic function and not a rational function,
$\infty$ is an essential singularity of $f$, and so is $a$. But this
is a contradiction to the fact that $f$ has no isolated
singularities other than poles in $\mathbb{C}$.

 \vskip2mm
\noindent \textbf{Case II.}  $P$ is an infinite set.  Let $\infty$ be
 the limit of all poles $\{w_k \}$ of $f.$ Note that $F_0$ contains only finitely many
accumulation points, that is, poles of $B.$    There is an integer
$k_0$ such that $w_k\in \mathcal{N}_2$ for $n\geq k_0.$ Then there
is a sequence $\{w_k':k\geq k_0\}$ in $ \mathcal{N}_1$ such that
$$B(w_k')=B(w_k ) \quad \mathrm{and} \quad w_k'\to a.$$
 Since the  only  accumulation points of
$F_0$ are poles of $B$,  these points $w_k'$ are isolated singularities of
$f$. Noting (\ref{equal2}), we have that  $w_k'$ are  poles of
$f$ because $w_k$ are poles of $f$. But $\{w_k'\}$ tends to the
finite point $a$, and thus $a$ is not an isolated singularity of
$f$. This is a contradiction to $f\in
\mathfrak{M}(\overline{\mathbb{D}}).$

Therefore, in both cases we derive a contradiction to finish the
proof.
\end{proof}

Some comments are in order. Quine \cite{Qu} showed that    each
polynomial  has FSI-decomposible property. Precisely, he proved that
a nonconstant polynomial can always be written as $p(z^m)(m\geq 1)$
where $p$ is a polynomial of FSI property. For  decomposition of
rational functions, we call the reader's attention to   \cite{Ri1,Ri2}.
\vskip1mm
 The following result gives some equivalent
conditions for an entire-symbol  Toeplitz operator to be totally
Abelian, and it follows from Theorems \ref{FSI2} and \ref{rational2}. The reader can consult    related work in
 \cite{Ti}. 
\begin{prop}Suppose $\phi(z)=\sum_{k=0}^\infty c_kz^k $ is a nonconstant entire function. Then  the
following are equivalent:
\begin{itemize}
\item[(1)]  $ T_\phi $ is totally Abelian;
\item [(2)]  $n   (\phi )=  \min\, \{\mathrm{wind} (\phi, \phi(a)): a\in \mathbb{D},\phi(a)\not\in f(\mathbb{T}) \}=1.$
\item [(3)] there is a point $w$ on $\mathbb{T}$ such that
 $\phi(w)$ is not a point of self-intersection;
\item [(4)]  $\phi$ has only finitely many points of self-intersection  on $\mathbb{T}$;
\item [(5)] there is a point $w$ in $\mathbb{D}$ such that
$\phi-\phi(w)$ has exactly one zero  in $\mathbb{D}$, counting
multiplicity.
\item [(6)] there is a point $\lambda   \in \mathbb{C}$ such that $\phi-\lambda$ has exactly
one zero   in $\mathbb{D}$, counting multiplicity;
\item [(7)] $\gcd \{c_k: c_k\neq 0\}=1$.
\end{itemize}
\end{prop}

\section{Some  examples}
~~~~This section provides some examples. Some of them are examples
of totally Abelian Toeplitz operators, and others will show that the
MWN Property is quite restricted even for  functions of good  ``smooth" property on $\mathbb{T}$.
\vskip1mm
We begin with a rational function in
$\mathfrak{R}(\overline{\mathbb{D}})$.
\begin{exam}Let $Q$ be a  polynomial without zero on
 $\overline{\mathbb{D}}$ and of  prime degree $q$.
Let $P$ be a nonconstant polynomial  satisfying $$\deg  P< q.
 $$Suppose  that  $P$ has at least one zero  in $  \mathbb{D}$ and let
 $R= \frac{P}{Q}$.  We will show that $T_R$ is totally Abelian. For this,  assume $R$ has $k$ zeros in $\mathbb{D}$,
counting multiplicity. We have
 \begin{equation}\label{ee}
 1\leq k \leq \deg P <q.
 \end{equation}
If $T_R$ were not totally Abelian, then by Corollary \ref{48} there
would be a finite Blaschke product $B$
 and a Moebius map  $\widetilde{R} $ such that
$$R=\widetilde{R}\circ B ,$$  and 
  $\mathrm{order} \,B$=$q$. But by $R=\widetilde{R}\circ B ,$ we  have
 $k\geq \mathrm{order}\, B=q$, which is a contradiction to  (\ref{ee}).  Therefore $T_R$ is totally Abelian.\end{exam}

The following two examples arise from the Riemann-zeta function and
the Gamma function. It is shown that under a translation or a
dilation of the variable, the corresponding Toeplitz operators are
totally Abelian.

\begin{exam}\label{exam1} The Riemann-zeta function $\zeta(z)$ is defined as the
analytic continuation of the following:
$$z\mapsto \sum_{n=1}^\infty \frac{1}{n^z},\, Re z>1.$$
It is a meromorphic function in $\mathbb{C}$ and the only pole is
$z=1 $. Write  \linebreak $f(z)=  \zeta( \frac{z}{2}),$ and then
$f(z)\in \mathfrak{M}(\overline{\mathbb{D}})$. We claim that $T_f$
is totally Abelian.

 For this, by Theorem \ref{mero} it suffices to show that there is no meromorphic function $g$ on $\mathbb{C}$
such that $f(z)=g(z^k) $ for some integer $k\geq 2.$   Otherwise,
taking $\omega\neq 1$ and $\omega^k=1$, we have  $f(\omega z)=f(z)$.
This gives $\zeta(\frac{\omega z}{2} ) =\zeta(\frac{z}{2}) $, and
thus
 $$\zeta( \omega z) =\zeta( z).$$
 Then $\zeta$
 has at least  two poles $1$ and $\overline{\omega}$. This is a contradiction.
Therefore, $T_f$ is totally Abelian.
\end{exam}

\begin{exam}\label{exam2} The Gamma function $\Gamma(z)$ is a meromorphic
function with only poles at non-positive integers
$$0,-1,-2,\cdots.$$
Let $f(z)=\Gamma(z+2) $ and  $f(z)\in
\mathfrak{M}(\overline{\mathbb{D}})$. Then $T_f$ is totally Abelian.

Otherwise, by Theorem \ref{mero}, there is a  function $g\in
\mathfrak{M}(\overline{\mathbb{D}})$ such that $f(z)=g(z^k) $ for
some integer $k\geq 2.$          Let $\omega\neq 1$ and
$\omega^k=1$,  and we have \linebreak  $f(z)=f(\omega z )$; that is,
$$\Gamma(z+2)=\Gamma(\omega z+2).$$The poles of $\Gamma(z+2)$
$$-2, -3,  \cdots  $$
must be  the  poles of $\Gamma(\omega z+2) $
 $$ -2\overline{\omega}, -3 \overline{\omega} ,  \cdots  .$$
 This is impossible.  Hence $T_f$ is totally Abelian. \end{exam}

Before continuing, recall that a Jordan domain is the interior of a
Jordan curve. We need  Caratheodory's theorem, which can be found in
a standard textbook of complex analysis, see \cite{Ah} for
example.
\begin{lem}\label{54}[Caratheodory's theorem] Suppose that  $\Omega$ is a Jordan domain.
Then the inverse Riemann mapping function $f$  from $\mathbb{D}$
onto $\Omega$  extends   to a 1-to-1 continuous function $F$ from
$\overline{\mathbb{D}}$ onto $\overline{\Omega}$. Furthermore, the
function $F $ maps $\mathbb{T}$ 1-to-1 onto $\partial \Omega$.
\end{lem}

In what follows, we provide  some  examples   to show that in
general a function $f$ in the disk algebra $A(\mathbb{D})$ may not
satisfy MWN Property, even if $f$ has good smoothness on
$\mathbb{T}$.

\begin{exam} First, we present an easy example of a function in $A(\mathbb{D}) $ with good smoothness on $\mathbb{T}$,
but not satisfying  MWN Property.
 Put $$\Omega_0=\{z:0< |z|<1,   \,  0<  \arg z <\pi\},$$
and write $g(z)=z^{8},\,z\in \Omega_0$. Let $\phi_0$ be a conformal
map from the unit disk $\mathbb{D}$ onto $\Omega_0$. Precisely,
write $u(z)=\sqrt{i \frac{z+1}{z-1}}$ with $\sqrt{1}=1$ and
$$ \phi_0(z)=\frac{1-2u(z)}{1+2u(z)},\, z\in \mathbb{D}. $$
and put $\phi_1=g\circ\phi_0.$ Note that $\phi_1-\phi_1(0)$ has
finitely many zeros in $\mathbb{D}$ and is away from zero on
$\mathbb{T}$. One   can then show that the inner part of
$\phi_1-\phi_1(0)$ is a finite Blaschke product, and hence
 $\phi_1 \in \mathcal{CT}(\mathbb{D})$.
\vskip1mm
For smoothness  of $\phi_1$, by Lemma \ref{54} we have that
$\phi_1\in A(\mathbb{D})$. In addition, by \label{36}
using Schwarz Reflection Principle we see that  except for at most three   points on $ \mathbb{T}$, 
  $\phi_1$   extends analytically across $ \mathbb{T}$.
\vskip1mm
However, $\phi_1$ does not satisfy MWN Property.  In fact, for each
point $a\in \mathbb{D}$, $\phi_1-\phi_1(a)$ has at least $3$ zeros
in $\mathbb{D}$. Thus,
$$n( \phi_1 )\geq 3.$$ On the other hand, since
$N(\phi_1,\mathbb{T})=1$, $\phi_1$ can not
be written as a function of a finite Blaschke product of order
larger than $1$. Then by Theorem \ref{Tm1}
   $\{T_{\phi_1}\}'=\{T_z\}'$. That is,  $\,
  b(\phi_1)=1.$
But $$n( \phi_1 ) >1 ,$$ forcing $n( \phi_1 )\neq b(\phi_1).$
\end{exam}

 Inspired by this example,
 put $\Omega_1=\{z\in \mathbb{C} : |z|<1 , |z-1|<1\}$, and let $h:\mathbb{D}\to \Omega$ be a conformal map.
 Note that by Caretheodory's theorem, $h$ extends continuously to a bijective map from $\overline{\mathbb{D}}$ onto $\overline{\Omega}$.
  Furthermore, noting that $\Omega_1$ has two cusp points, one can show that except for two cusp points,
  $h$ can be analytically extended across $ \mathbb{T}$, as well as $h^9$. Also, $\{T_{h^9}\}'=\{T_z\}'$.

\vskip2mm
  The next example
     shows that
  Theorems  \ref{rational} and \ref{ratgroup} are
  restricted.
\begin{exam}     \label{exam3}
  By Schwarz-Christoffel formula, one can construct a conformal map $f$ form the
upper plane onto the rectangle   $\Omega$ with vertices \linebreak
 $\{-\frac{K}{2},\frac{K}{2},\frac{K}{2}+iK', -\frac{K}{2}+iK' \},$ where $K,K'>0.$
  Precisely, it is defined by
$$f(z)=C \int_0^z \frac{1}{\sqrt{( \lambda^2-1)( \lambda^2-t^2)}}d\lambda, \,z\in\Omega,$$
where $\sqrt{1}=1$, $C >0$ and $t$ is a parameter in $ (0,1)$
\cite[Section 2.5]{DT}. We can specialize    $K'=2k\pi$
 for some integer $k\geq 100.$

 Define $h(z)=\exp(z-\frac{K}{2}),\, z\in \mathbb{C}$ and let $g(z)$ be a conformal map from the unit disk onto the upper plane.
Write $$\phi(z)=h\circ f\circ g(z), \, z\in \mathbb{D}.$$ It is not
difficult to see that
$$n (\phi)\geq k,$$
 and for each $\xi\in  \mathbb{T} ,$
 $N(\phi-\phi(\xi), \mathbb{T})\geq 2.$ Moreover, we have $N(\phi)=2.$
\vskip1mm
Next we show  that $o(\phi)=1.$ For this, note that $f\circ g$ maps
the unit disk $\mathbb{D}$ conformally onto the   rectangle
$\Omega$, and  $f\circ g$ extends to a continuous bijection from
$\overline{\mathbb{D}}$ onto $\overline{\Omega}$, and $f\circ
g(\mathbb{T})=\partial \Omega$. Thus by definition of $o(\phi),$ the
assertion
 $o(\phi)=1 $ is equivalent to that    the only continuous map $\rho: \partial \Omega\to \partial \Omega$
satisfying $h( \rho) =h$ is the identity map. For this, let
$$h(\rho(z))=h(z), z\in \partial \Omega .$$
Then for each $z$ in $\partial \Omega$, $\rho(z)=z+2k(z) \pi i$ for
some integer $k(z).$ But $\rho$ is continuous, forcing $k(z)$ to be
a constant integer $k$. Hence $$\rho(z)=h(z)+2k \pi i, z\in
\partial \Omega. $$ Since $\rho(\partial \Omega)\subseteq
\partial \Omega$, we have $k=0$ and   $\rho$ is the identity
map, forcing \linebreak $o(\phi)=1.$

 Since $b(\phi)\,|\, o(\phi),$
 $b(\phi)=1.$
  By Theorem \ref{Tm1},     $T_{\phi }$ is totally Abelian.
 But for this function $\phi $ we have
  $$b(\phi) = o(\phi)<N(\phi)<n(\phi).$$
\end{exam}

  We conclude this section by showing that the function $\phi$ defined in Example 6.6
  has good smoothness property on $\mathbb{T}$.
   Rewrite $h(z)= v(z)^{2n}$ where $v(z)=\exp[\frac{1}{2k}(z-\frac{K}{2})].$ Note that
 $v\circ f\circ g$ defines a conformal map from $\mathbb{D}$ onto the domain
$$ \{z\in \mathbb{C}:\exp(\frac{-K}{2k})<|z|<1, \, \arg z \in (0,\pi)\},$$
whose boundary contains only four ``cusp points". By Lemma \ref{54}
we have $v\circ f\circ g\in A(\mathbb{D})$,
 and by Schwarz Reflection Principle  $v\circ f\circ g$ extends analytically across $\mathbb{T}$ except for these
 cusp points. The same is true for $\phi.$

\vskip2mm
 \noindent \textbf{Acknowledgement} This work is  partially supported by
 National Nature Science Foundation of China, 
       and by  Shanghai Center for Mathematical Sciences.

\vskip3mm

\noindent{Hui Dan, School of Mathematical Sciences, Fudan
University, Shanghai, 200433, China, E-mail: 15110180030@fudan.edu.cn

\noindent Kunyu Guo, School of Mathematical Sciences, Fudan University,
 Shanghai, 200433, China, E-mail: kyguo@fudan.edu.cn

\noindent Hansong Huang, Department of Mathematics, East China
University of
 Science and Technology, Shanghai, 200237, China, E-mail:
hshuang@ecust.edu.cn }

\end{document}